\documentclass[reqno]{amsart}

\usepackage[foot]{amsaddr}
\usepackage{graphicx,enumerate,nicefrac,bm}
\usepackage[dvipsnames]{xcolor}
\usepackage{cite}
\usepackage{amsmath, amssymb}
\usepackage{tikz}\usetikzlibrary{matrix}
\usepackage{a4wide}
\usepackage{subfig}
\usepackage{changes}
\usepackage{mathtools}
\usepackage{algorithm,algpseudocode,tabularx}
\usepackage{comment}

\makeatletter
\newcommand{\multiline}[1]{%
  \begin{tabularx}{\dimexpr\linewidth-\ALG@thistlm}[t]{@{}X@{}}
    #1
  \end{tabularx}
}
\makeatother

\newcommand{\ukappa}{\underline{\kappa}}
\newcommand{\utau}{\underline{\tau}}

\newcommand{\uv}{\mathbf{u}}
\newcommand{\vv}{\mathbf{v}}
\newcommand{\wv}{\mathbf{w}}
\newcommand{\eu}{\underline{e}(\mathbf{u})}
\newcommand{\ev}{\underline{e}(\mathbf{v})}
\newcommand{\A}{\mathsf{A}}
\newcommand{\E}{\mathsf{E}}

\newcommand{\x}{\bm{x}}
\newcommand{\dx}{\,\mathsf{d}\bm{x}}

\newcommand{\HS}{\mathrm{H}^1_0(\Omega)}

\newcommand{\Lom}{\mathrm{L}^{2}(\Omega)}
\newcommand{\e}[1]{\underline{e}(\mathbf{#1})}
\newcommand{\ef}[1]{\underline{e}(#1)}

\newcommand{\norm}[1]{\left\|#1\right\|}
\newcommand{\nnn}[1]{{\left\vert\kern-0.25ex\left\vert\kern-0.25ex\left\vert #1 
    \right\vert\kern-0.25ex\right\vert\kern-0.25ex\right\vert}}

\DeclareMathOperator*{\esssup}{ess\,sup}
\DeclareMathOperator*{\essinf}{ess\,inf}

\newtheorem{theorem}{Theorem}[section]
\newtheorem{lemma}[theorem]{Lemma}
\newtheorem{proposition}[theorem]{Proposition}

\theoremstyle{definition}

\newtheorem{remark}[theorem]{Remark}

\title[Convergence rate for shear-thinning fluids]{On the convergence rate of the Ka\v{c}anov scheme for shear-thinning fluids}

\author[P.~Heid]{Pascal Heid}
\email{pascal.heid@maths.ox.ac.uk}

\author[E.~S\"{u}li]{Endre S\"{u}li}
\email{endre.suli@maths.ox.ac.uk}

\address{Mathematical Institute, University of Oxford, Woodstock Road, Oxford OX2 6GG, UK}

\thanks{PH acknowledges the financial support of the Swiss National Science Foundation (SNF), Project No. P2BEP2\underline{\space}191760.}



\begin{document}

\begin{abstract}
We explore the convergence rate of the Ka\v{c}anov iteration scheme for different models of 
shear-thinning fluids, including Carreau and power-law type explicit quasi-Newtonian constitutive laws. It is shown that the energy difference contracts along the sequence generated by the iteration. In addition, an \emph{a posteriori} computable contraction factor is proposed, which  improves, on finite-dimensional Galerkin spaces, previously derived bounds on the contraction factor in the context of the power-law model. Significantly, this factor is shown to be independent of the choice of the cut-off parameters whose use was proposed in the literature for the Ka\v{c}anov iteration applied to the power-law model. Our analytical findings are confirmed by a series of numerical experiments.
\end{abstract}



\maketitle

\section{Introduction}

In this work, we focus on the iterative solution of nonlinear partial differential equations that arise in models of steady flows of incompressible shear-thinning fluids, including models with explicit constitutive relations of Carreau and power-law type. 
In particular, we consider the following quasi-Newtonian fluid flow problem: find $(\mathbf{u},p)$ such that
\begin{align} \label{eq:strong}
\begin{split}
-\nabla \cdot \{\mu(\x,|\underline{e}(\mathbf{u})|^2) \underline{e}(\mathbf{u})\}+\nabla p &= \mathbf{f} \quad \text{in } \Omega, \\
\nabla \cdot \mathbf{u}&=0  \quad \text{in } \Omega, \\
\mathbf{u}&=0 \quad \text{on } \partial \Omega,
\end{split}
\end{align}
where $\Omega \subset \mathbb{R}^d$, $d \in \{2,3\}$, is a bounded Lipschitz domain, the source term $\mathbf{f} \in \Lom^d$ is a given external force, $\mathbf{u}$ is the velocity vector, $p$ denotes the pressure, and $\underline{e}(\mathbf{u})$ is the $d \times d$ rate-of-strain tensor defined by 
\[
e_{ij}(\mathbf{u}):=\frac{1}{2} \left(\frac{\partial u_i}{\partial x_j}+\frac{\partial u_j}{\partial x_i}\right),\qquad i,j=1,\ldots,d.
\]
Here, $|\underline{e}(\mathbf{u})|$ denotes the Frobenius norm of $\underline{e}(\mathbf{u})$, and the (real-valued) viscosity coefficient $\mu$ is assumed to satisfy the following structural assumptions:
\begin{enumerate}[({A}1)]
\item $\mu \in C(\overline{\Omega} \times \mathbb{R}_{\geq 0})$ and it is differentiable in the second variable;
\item There exist constants $m_\mu,M_\mu >0$ such that
\begin{align} \label{eq:a2}
m_\mu (t-s) \leq \mu(\x,t^2)t-\mu(\x,s^2)s \leq M_\mu(t-s), \qquad t \geq s \geq 0, \quad \x \in \overline{\Omega};
\end{align}
\item $\mu$ is decreasing in the second variable, i.e., $\mu'(\x,t) \leq 0$ for all $t \geq 0$ and all $\x \in \overline{\Omega}$.
\end{enumerate} 
The assumption (A3) asserts that the viscosity decreases with increasing strain rate, in line with our assumption that the fluid under consideration is shear-thinning. Moreover, (A2) immediately implies that $\mu$ is bounded from above and below; indeed, by setting $s=0$, we obtain
\begin{align} \label{eq:mubd}
m_\mu \leq \mu(\x,t) \leq M_\mu \qquad \text{for all } \x \in \overline{\Omega},  \, t \geq 0.
\end{align}
The bounds $m_\mu$ and $M_\mu$ are, in general, closely related to the infinite and zero shear viscosity plateau, respectively. In the sequel, the dependence of $\mu$ on $\x \in \Omega$ will be suppressed. 

Upon defining $V:=\{\uv \in \HS^d: \nabla \cdot \mathbf{u}=0\}$, the weak formulation of~\eqref{eq:strong} is  as follows:
\begin{align} \label{eq:directweak}
\text{find} \ \uv \in V \ \text{such that} \qquad \int_\Omega \mu(|\underline{e}(\mathbf{u})|^2)\eu:\ev \dx = \int_\Omega \mathbf{f} \cdot \mathbf{v} \dx \qquad \text{for all } \vv \in V,
\end{align}
where $\eu:\ev$ denotes the Frobenius inner product of $\eu$ and $\ev$; we refer to~\cite[\S 2]{BarrettLiu:1993} for more details concerning the weak formulation~\eqref{eq:directweak}. The space $V$ is endowed with the inner product
\begin{align} \label{eq:Vip}
(\uv,\vv)_V=\int_\Omega \e{u}:\e{v} \dx, \qquad \uv,\vv \in V,
\end{align}
and the induced norm $\nnn{\uv}_\Omega^2=(\uv,\uv)_V$, $\uv \in V$. We emphasize that 
\begin{align*} 
\frac{1}{2} \sum_{i=1}^d \int_\Omega |\nabla u_i|^2 \dx \leq \int_\Omega |\eu|^2 \dx \leq \sum_{i=1}^d \int_\Omega |\nabla u_i|^2\dx \qquad \text{for all } \uv \in V,
\end{align*}
i.e., the norm $\nnn{\cdot}_\Omega$ is equivalent to the standard norm on $\HS^d$; the first inequality is a special case of Korn's inequality (see, e.g.,~inequality (1.7) in \cite{NeffPauly:2015}), while the second can be easily verified by invoking the Cauchy--Schwarz inequality. In particular, $V$ endowed with the inner product of~\eqref{eq:Vip} and induced norm $\nnn{\cdot}_\Omega$ is a Hilbert space. 

The weak form~\eqref{eq:directweak} of the boundary-value problem under consideration is known to have a unique solution $\uv^\star \in V$, which will be shown, nonetheless, in \S\ref{sec:uniquesol}; moreover, this element $\uv^\star \in V$ is the unique minimiser of the energy functional 
\begin{align} \label{eq:Edeff}
\E(\uv):= \int_\Omega \varphi(|\eu|^2) \dx-\int_\Omega \mathbf{f} \cdot \mathbf{u} \dx, \qquad \uv \in V,
\end{align}
where
\begin{align*}
\varphi(s):=\frac{1}{2}\int_0^s \mu(t)\, \mathrm{d}t.
\end{align*}
Indeed, a straightforward calculation reveals that, for a given $\uv \in V$,
\begin{align} \label{eq:Efrechet}
\E'(\uv)(\vv)=\int_\Omega \mu(|\eu|^2)\eu:\ev \dx - \int_{\Omega} \mathbf{f} \cdot \mathbf{v} \dx, \qquad \mathbf{v} \in V,
\end{align}
where $\E'$ denotes the G\^{a}teaux derivative; we refer to~\cite[Prop.~2.1]{Baranger:90} for details. In particular, the weak formulation~\eqref{eq:directweak} is the Euler--Lagrange equation for the minimisation of $\E$ over $V$.

A prominent iterative solver for the nonlinear problem~\eqref{eq:directweak} is Ka\v{c}anov's scheme, which, in simple terms, fixes the nonlinearity at the previous iterate: 
for a given $\uv^n \in V$ find $\uv^{n+1} \in V$ such that
\begin{align} \label{eq:kacanovintro}
\int_\Omega \mu(|\underline{e}(\mathbf{u}^n)|^2)\ef{\uv^{n+1}}:\ev \dx = \int_\Omega \mathbf{f} \cdot \mathbf{v} \dx \qquad \text{for all } \vv \in V,\quad n=0,1,\ldots,
\end{align} 
where $\uv^0 \in V$ is an arbitrary initial guess. Early references concerning this iterative method include \cite{kacur:1968}, where it was used to compute a stationary magnetic field in nonlinear media,
and \cite{fucik:1973}, where the convergence of the Ka\v{c}anov iteration was investigated in the context of Galerkin methods; Fu\v{c}\'{\i}k, Kratochv\'{\i}l and Ne\v{c}as point in their work \cite{fucik:1973} to pages 369--370 of Michlin's 1966 monograph \cite{michlin:1966} for a description of the iterative method
introduced by Ka\v{c}anov in \cite{kacanov:1959},
in the context of variational methods for plasticity problems.
Ka\v{c}anov's iteration scheme has been, by now, carefully examined; see, e.g., the monographs~\cite[\S4.5]{Necas:86} and~\cite[\S25.14]{Zeidler:90}, or the papers~\cite{HaJeSh:97,GarauMorinZuppa:11,HeidWihler:19v2}. More recently, it was shown in the articles~\cite{HeidPraetoriusWihler:2020} and~\cite{Diening:2020} that the energy $\E$ from~\eqref{eq:Edeff} contracts along the sequence generated by the Ka\v{c}anov iteration~\eqref{eq:kacanovintro}. Indeed, the first of these two papers established the energy contraction for a more general iteration scheme, and the latter focuses on the Ka\v{c}anov scheme for a `relaxed $p$-Poisson problem' involving a truncation of the nonlinearity from below and from above using a pair of positive cut-off parameters $\varepsilon_{-}$ and $\varepsilon_{+}$. The derived upper bound on the contraction factor depends on the quotient $\nicefrac{m_\mu}{M_\mu}$ involving $\varepsilon_{-}$ and $\varepsilon_{+}$, and may be extremely close to 1 in certain situations; interestingly, this unfavourable predicted dependence of the contraction factor on the ratio $\nicefrac{m_\mu}{M_\mu}$ has not been observed in numerical experiments.
It is this mismatch between the observed behaviour of the method and the rather more pessimistic results of the analysis reported in the literature that motivated the work reported herein. 


We will establish an improved upper bound on the contraction factor of the Ka\v{c}anov iteration for a general class of shear-thinning fluids. The resulting bound will then be further examined for fluids obeying either the Carreau law or a relaxed power-law. It will be shown that for (finite-dimensional) Galerkin approximations of the relaxed power-law model it is the power-law exponent, rather than the ratio $\nicefrac{m_\mu}{M_\mu}$, that is responsible for the rate of convergence of the iteration. Specifically, we will show that the contraction factor of the iteration on finite-dimensional spaces is independent of the choice of the lower and upper cut-off parameters featuring in the so-called relaxed Ka\v{c}anov iteration, where a truncation of the power-law nonlinearity from below and above is carried out by means of these two positive truncation parameters. To the best of our knowledge the proof of such a result was an open question in the literature.


The paper is structured as follows. 
In \S\ref{sec:uniquesol} we will show that the weak formulation~\eqref{eq:directweak} of the problem under consideration has a unique solution, which, in turn, is the unique minimiser of $\E$ in $V$. The proof is based on auxiliary results, which will also be decisive for the derivation of the contraction factor in \S\ref{sec:energycontraction}. In \S\ref{sec:numericalexp} we will perform a series of numerical experiments, which confirm our theoretical results. The paper closes with concluding remarks recorded in \S\ref{sec:conclusions}.

\section{Existence and uniqueness of the solution} \label{sec:uniquesol}

We will show in this section that the weak formulation~\eqref{eq:directweak} has a unique solution. To this end we define, for given $\uv \in V$, the linear operator $\A[\uv]: V \to V^\star$, where $V^\star$ denotes the dual space of $V$, by
\begin{align} \label{eq:Aoperator}
\A[\uv](\vv)(\wv):= \int_\Omega \mu(|\eu|^2) \ev:\e{w} \dx, \qquad \vv,\wv \in V,
\end{align}
and the linear form $\ell \in V^\star$ by
\begin{align*} 
\ell(\wv):=\int_\Omega \mathbf{f} \cdot \mathbf{w} \dx, \qquad \wv \in V.
\end{align*}
In terms of these, the weak formulation~\eqref{eq:directweak} can be restated in the following equivalent form:
\begin{align} \label{eq:weakoperatorform}
\text{find} \ \uv \in V \ \text{such that} \qquad \A[\uv](\uv)(\vv)=\ell(\vv) \qquad \text{for all } \vv \in V,
\end{align}
and the Ka\v{c}anov iteration~\eqref{eq:kacanovintro} takes the form: given $\uv^0 \in V$, 
\begin{align} \label{eq:kacanov}
\text{find} \ \uv^{n+1} \in V \ \mbox{such that} \quad \A[\uv^n](\uv^{n+1})(\vv)=\ell(\vv) \qquad \text{for all } \vv \in V, \quad n=0,1,\ldots.
\end{align} 
By~\eqref{eq:Efrechet} and the definitions of $\A$ and $\ell$ we further have that
\begin{align} \label{eq:Efrechet2}
\E'(\uv)=\A[\uv](\uv)-\ell.
\end{align}

We will now show that the operator $\uv \mapsto \A[\uv](\uv)$ is Lipschitz continuous and strongly monotone, since, in that case, the theory of strongly monotone operators implies that the weak equation~\eqref{eq:weakoperatorform} has a unique solution $\uv^\star \in V$; see, e.g., \cite[\S3.3]{Necas:86} or~\cite[\S25.4]{Zeidler:90}. For the proof of Lipschitz continuity and strong monotonicity of the operator $\uv \mapsto \A[\uv](\uv)$ we require the following result, which, as well as its proof, is largely borrowed from~\cite[Lem.~3.1]{BarrettLiu:1993}. However, we place emphasis on sharp bounds, since these will be crucial for our convergence analysis below, leading to the improved factors appearing in~\eqref{eq:mukappaconstU} and~\eqref{eq:mukappaconstl}.

\begin{lemma}
Let $\mu$ satisfy the assumptions {\rm (A1)--(A3)} and define $\xi(t):=\mu(t^2)t$, $t \geq 0$. Then, for any $\underline{\kappa},\underline{\tau} \in \mathbb{R}^{d\times d}$, the following inequalities hold:
\begin{align} \label{eq:mukappaU}
\left|\mu(|\underline{\kappa}|^2)\underline{\kappa}-\mu(|\underline{\tau}|^2)\underline{\tau}\right|^2 \leq C(\underline{\kappa},\underline{\tau}) |\underline{\kappa}-\underline{\tau}|^2 \leq 3 M_\mu^2 |\underline{\kappa}-\underline{\tau}|^2
\end{align}
and 
\begin{align} \label{eq:mukappal}
(\mu(|\underline{\kappa}|^2)\underline{\kappa}-\mu(|\underline{\tau}|^2)\underline{\tau}):(\underline{\kappa}-\underline{\tau}) \geq c(\underline{\kappa},\underline{\tau})|\underline{\kappa}-\underline{\tau}|^2 \geq m_\mu  |\underline{\kappa}-\underline{\tau}|^2,
\end{align}
where
\begin{align} 
C(\underline{\kappa},\underline{\tau})&:=\left(\sup_{t \in (0,1)} \xi'(|\underline{\kappa}|+t(|\underline{\tau}|-|\underline{\kappa}|)) \right)^2+2\mu(|\underline{\kappa}|^2)\mu(|\underline{\tau}|^2), \label{eq:mukappaconstU} \\ c(\underline{\kappa},\underline{\tau})&:=\inf_{t \in (0,1)} \xi'(|\underline{\kappa}|+t(|\underline{\tau}|-|\underline{\kappa}|)). \label{eq:mukappaconstl}
\end{align}
\end{lemma}

\begin{proof}
We will only prove~\eqref{eq:mukappal}, since the contraction factor will strongly rely on this bound, but not on~\eqref{eq:mukappaU}. For the proof of the latter, we refer to~\cite{BarrettLiu:1993}. 

A simple and straightforward calculation reveals that
\begin{align} 
(\mu(|\underline{\kappa}|^2)\underline{\kappa}-\mu(|\underline{\tau}|^2)\underline{\tau}):(\underline{\kappa}-\underline{\tau}) &= \mu(|\ukappa|^2)|\ukappa|^2+\mu(|\utau|^2)|\utau|^2-\mu(|\ukappa|^2)\ukappa:\utau-\mu(|\utau|^2)\utau:\ukappa \nonumber \\ 
&=\left(\mu(|\ukappa|^2)|\ukappa|-\mu(|\utau|^2)|\utau|\right)(|\ukappa|-|\utau|) \label{eq:firstsummand} \\
& \quad + (\mu(|\ukappa|^2)+\mu(|\utau|^2))(|\ukappa||\utau|-\ukappa:\utau). \label{eq:secondsummand}
\end{align}
We note that the summand in~\eqref{eq:firstsummand} can be written as $(\xi(|\ukappa|)-\xi(|\utau|))(|\ukappa|-|\utau|)$, since $\xi(t)=\mu(t^2)t$ for $t \geq 0$. Then, the mean value theorem implies that
\begin{align} \label{eq:firstsumbound}
(\xi(|\ukappa|)-\xi(|\utau|))(|\ukappa|-|\utau|) \geq \inf_{t \in (0,1)} \xi'(|\ukappa|+t(|\utau|-|\ukappa|))(|\ukappa|-|\utau|)^2=c(\ukappa,\utau)(|\ukappa|-|\utau|)^2.
\end{align}
Furthermore, since $\mu'(t) \leq 0$ for all $t \geq 0$ by (A3), and, in turn, $\xi'(t)=\mu(t^2)+2t^2 \mu'(t) \leq \mu(t^2)$, we find that 
\[c(\ukappa,\utau)=\inf_{t \in (0,1)} \xi'(|\ukappa|+t(|\utau|-|\ukappa|)) \leq \min\left\{\mu(|\ukappa|^2),\mu(|\utau|^2)\right\}\!.\]
Consequently, the summand in~\eqref{eq:secondsummand} can be bounded from below by
\begin{align} \label{eq:secondsumbound}
(\mu(|\ukappa|^2)+\mu(|\utau|^2))(|\ukappa||\utau|-\ukappa:\utau) \geq 2c(\ukappa,\utau)(|\ukappa||\utau|-\ukappa:\utau),
\end{align}
since $|\ukappa||\utau|-\ukappa:\utau \geq 0$ by the Cauchy--Schwarz inequality. Hence, using the established bounds~\eqref{eq:firstsumbound} and~\eqref{eq:secondsumbound} for the summands in~\eqref{eq:firstsummand} and~\eqref{eq:secondsummand}, respectively, yields
\begin{align*}
(\mu(|\underline{\kappa}|^2)\underline{\kappa}-\mu(|\underline{\tau}|^2)\underline{\tau}):(\underline{\kappa}-\underline{\tau}) &\geq c(\ukappa,\utau)\left( (|\ukappa|-|\utau|)^2+2\left(|\ukappa||\utau|-\ukappa:\utau\right) \right) \\
&=c(\ukappa,\utau)\left(|\ukappa|^2+|\utau|^2-2 \ukappa:\utau\right) \\
&=c(\ukappa,\utau)|\ukappa-\utau|^2.
\end{align*}
Finally we note that dividing~\eqref{eq:a2} by $(t-s)$ and taking the limit $s \to t$ yields that $m_\mu \leq \xi'(t) \leq M_\mu$ for all $t \geq 0$, i.e., $c(\ukappa,\utau) \geq m_\mu$.
\end{proof}

Now we are ready to show that $\A$, cf.~\eqref{eq:Aoperator}, is strongly monotone and Lipschitz continuous, which implies the unique solvability of~\eqref{eq:directweak}. 

\begin{proposition} \label{lem:properties}
Let $\A$ be defined as in~\eqref{eq:Aoperator}, with $\mu$ satisfying {\rm (A1)--(A3)}.
\begin{enumerate}[(a)]
\item For given $\uv \in V$, $\A[\uv](\cdot)(\cdot)$ is a uniformly bounded and coercive, symmetric bilinear form on $V \times V$. In particular, the following inequalities hold:
\begin{align*} 
\A[\uv](\vv)(\wv) \leq M_\mu \nnn{\vv}_\Omega \nnn{\wv}_\Omega
\end{align*}
and 
\begin{align} \label{eq:coercive}
\A[\uv](\vv)(\vv) \geq m_\mu \nnn{\vv}_\Omega^2
\end{align}   
for any $\uv,\vv,\wv \in V$.
\item The mapping $\uv \mapsto \A[\uv](\uv)$ is Lipschitz continuous with
\begin{align} \label{eq:lipschitz}
\A[\mathbf{u}](\mathbf{u})(\mathbf{w})-\A[\mathbf{v}](\mathbf{v})(\mathbf{w}) \leq \sqrt{3} M_\mu \nnn{\uv-\vv}_\Omega \nnn{\wv}_\Omega, \qquad \uv,\vv,\wv \in V
\end{align}
and strongly monotone with
\begin{align} \label{eq:smonotone}
\A[\mathbf{u}](\mathbf{u})(\mathbf{u}-\mathbf{v})-\A[\mathbf{v}](\mathbf{v})(\mathbf{u}-\mathbf{v}) \geq m_\mu \nnn{\uv-\vv}_\Omega^2, \qquad \uv,\vv \in V.
\end{align} 
Consequently, the problem~\eqref{eq:directweak} has a unique solution $\uv^\star \in V$.
\end{enumerate} 
\end{proposition}

\begin{proof}
Ad (a): By invoking the definition of $\A$, cf.~\eqref{eq:Aoperator}, the boundedness of the viscosity coefficient $\mu$, cf.~\eqref{eq:mubd}, and applying the Cauchy--Schwarz inequality twice, we obtain 
\begin{align*}
\A[\uv](\vv)(\wv) &= \int_\Omega \mu(|\eu|^2) \ev:\e{w} \dx \\
& \leq M_\mu \int_\Omega |\ev:\e{w}| \dx\\
& \leq M_\mu \left(\int_\Omega |\ev|^2 \dx\right)^{\nicefrac{1}{2}} \left(\int_\Omega |\e{w}|^2 \dx\right)^{\nicefrac{1}{2}} \\
& = M_\mu \nnn{\vv}_\Omega \nnn{\wv}_\Omega.
\end{align*}
Similarly, the inequality~\eqref{eq:mubd} implies the uniform coercivity~\eqref{eq:coercive}.

Ad (b): The definition of $\A$, cf.~\eqref{eq:Aoperator}, and the Cauchy--Schwarz inequality yield
\begin{align*}
\A[\mathbf{u}](\mathbf{u})(\mathbf{w})-\A[\mathbf{v}](\mathbf{v})(\mathbf{w})&=\int_\Omega \left[\mu(|\eu|^2)\eu-\mu(|\ev|^2)\ev\right]:\e{w} \dx\\
& \leq \int_\Omega |\mu(|\eu|^2)\eu-\mu(|\ev|^2)\ev||\e{w}| \dx .
\end{align*}
Hence, by~\eqref{eq:mukappaU} and the linearity of $\e{\cdot}$, this leads to
\begin{align*}
\A[\mathbf{u}](\mathbf{u})(\mathbf{w})-\A[\mathbf{v}](\mathbf{v})(\mathbf{w}) \leq \sqrt{3} M_\mu \int_\Omega |\ef{\uv-\vv}||\e{w}| \dx.
\end{align*}
Applying once more the Cauchy--Schwarz inequality implies that
\begin{align*}
\A[\mathbf{u}](\mathbf{u})(\mathbf{w})-\A[\mathbf{v}](\mathbf{v})(\mathbf{w})& \leq \sqrt{3} M_\mu \left(\int_\Omega |\e{u-v}|^2 \dx\right)^{\nicefrac{1}{2}} \left(\int_\Omega |\e{w}|^2 \dx\right)^{\nicefrac{1}{2}} \\ &= \sqrt{3} M_\mu \nnn{\mathbf{u}-\mathbf{v}}_\Omega \nnn{\mathbf{w}}_\Omega.
\end{align*}
Similarly, by the definition of $\A$, cf.~\eqref{eq:Aoperator}, and~\eqref{eq:mukappal} we obtain
\begin{align*}
\A[\uv](\uv)(\uv-\vv)-\A[\vv](\vv)(\uv-\vv)&=\int_\Omega \left(\mu(|\eu|^2)\eu-\mu(|\ev|^2)\ev\right):(\e{u}-\e{v}) \dx \\
& \geq m_\mu \int_\Omega |\e{u-v}|^2 \dx \\
& = m_\mu \nnn{\uv-\vv}_\Omega^2.
\end{align*}
The existence and uniqueness of a solution to the equation~\eqref{eq:directweak} now follows from the theory of monotone operators, cf.~\cite[\S3.3]{Necas:86} or~\cite[\S25.4]{Zeidler:90}.
\end{proof}

\begin{remark}
Since~\eqref{eq:directweak} is the Euler--Lagrange equation of the minimisation problem
\begin{align*} 
\min_{\uv \in V} \E(\uv),
\end{align*} 
the above proposition yields that $\uv^\star \in V$ is the unique minimiser of the functional $\E$. Moreover, the Ka\v{c}anov scheme~\eqref{eq:kacanov} is well defined thanks to Proposition~\ref{lem:properties} (a) and the Lax--Milgram theorem.
\end{remark}

\section{Energy contraction} \label{sec:energycontraction}

In this section, we will show that the energy error, given by $\E(\uv^n)-\E(\uv^\star)$, contracts along the sequence $\{\uv^n\}$ generated by the Ka\v{c}anov iteration~\eqref{eq:kacanov}. To this end, we need an auxiliary result.

\begin{lemma} 
Let $\A$ and $\E$ be defined as in \eqref{eq:Aoperator} and \eqref{eq:Edeff}, respectively, with $\mu$ satisfying {\rm (A1)--(A3)}. Then
\begin{align} \label{eq:keyin}
\E(\uv)-\E(\vv) \geq \frac{1}{2} \A[\uv](\uv)(\uv)-\frac{1}{2}\A[\uv](\vv)(\vv)-\ell(\uv)+\ell(\vv), \qquad \uv,\vv \in V.
\end{align}
\end{lemma}

This result is well-known for the Ka\v{c}anov iteration in the given setting, and the proof can be found, e.g., in \cite[\S 25.12]{Zeidler:90} or~\cite[\S 4.5]{Necas:86}. However, as it is stated in a slightly different form in those references, and also for the sake of completeness, we will include the proof of this statement nonetheless.

\begin{proof}
It can be shown that
\[\varphi(t)-\varphi(s) \geq \frac{1}{2} \mu(t) (t-s), \qquad t,s \geq 0,\]
see, e.g.,~\cite[\S 5.1]{HeidWihler:19v2}, and therefore
\begin{align*}
\int_\Omega \varphi(|\eu|^2)-\varphi(|\ev|^2) \dx &\geq \frac{1}{2} \int_\Omega \mu(|\eu|^2)(|\eu|^2-|\ev|^2)\dx \\
&= \frac{1}{2} \left(\A[\uv](\uv)(\uv)-\A[\uv](\vv)(\vv)\right),
\end{align*} 
for any $\uv,\vv \in V$. Hence, by the definition of $\E$, cf.~\eqref{eq:Edeff}, we find that
\begin{align*}
\E(\uv)-\E(\vv)&=\int_\Omega \varphi(|\eu|^2) \dx -\int_\Omega \varphi(|\ev|^2) \dx -\ell(\uv)+\ell(\vv)\\
 &\geq \frac{1}{2} \A[\uv](\uv)(\uv)-\frac{1}{2}\A[\uv](\vv)(\vv)-\ell(\uv)+\ell(\vv),
\end{align*}
which completes the proof of the claim.
\end{proof}

Now we are in a position to prove the contraction of the energy along the sequence generated by the Ka\v{c}anov scheme~\eqref{eq:kacanov}. We note that similar results can be found, e.g., in~\cite[Prop.~2.1]{HeidPraetoriusWihler:2020} or ~\cite[Cor.~19]{Diening:2020}.  

\begin{theorem} \label{thm:globalcontraction}
Assume that {\rm (A1)--(A3)} hold and let $\E$ be defined as in~\eqref{eq:Edeff}. Then, the energy error contracts along the sequence $\{\uv^n\}$ generated by the Ka\v{c}anov iteration~\eqref{eq:kacanov} in the sense that 
\begin{align*} 
0\leq\E(\uv^{n+1})-\E(\uv^\star) \leq q(n) \left(\E(\uv^n)-\E(\uv^\star)\right)\!,
\end{align*}
where
\begin{align} \label{eq:contractionfactor}
q(n):=1-\frac{1}{4} \left\{\esssup_{\x \in \Omega} \frac{\mu(|\ef{\uv^n}|^2)}{\inf_{t \in (-1,1)}\xi'(\max\left\{0,|\ef{\uv^\star}|+t|\ef{\uv^n}-\ef{\uv^\star}|\right\})}\right\}^{-1},
\end{align}
and $\xi(t)=\mu(t^2)t$ for $t \geq 0$.
\end{theorem}

\begin{proof}
We largely proceed along the lines of~\cite{Diening:2020}. However, as we want to improve the contraction factor from this reference and, in addition, remove any unknown constants, some non-trivial modifications are necessary in the second part of the proof. 

Let us define the real-valued function $\psi(t):=\E(\uv^\star+t(\uv^n-\uv^\star))$, $t \in [0,1]$. Then, by invoking the fundamental theorem of calculus, we obtain
\begin{align*}
\E(\uv^n)-\E(\uv^\star)=\int_0^1 \psi'(t) \, \mathrm{d} t.
\end{align*}
It can be shown that $\psi'(t), \, t \in [0,1]$, is increasing, and therefore
\begin{align*}
\E(\uv^n)-\E(\uv^\star) \leq \psi'(1)= \E'(\uv^n)(\uv^n-\uv^\star)=\A[\uv^n](\uv^n)(\uv^n-\uv^\star) - \ell(\uv^n-\uv^\star).
\end{align*}
Moreover, by the definition of the Ka\v{c}anov scheme~\eqref{eq:kacanov}, we have that 
\[ \A[\uv^n](\uv^{n+1})(\uv^n-\uv^\star)=\ell(\uv^n-\uv^\star).\] 
Consequently, the above inequality becomes
\begin{align*}
\E(\uv^n)-\E(\uv^\star) &\leq \A[\uv^n](\uv^n-\uv^{n+1})(\uv^n-\uv^\star) =\int_\Omega \mu(|\ef{\uv^n}|^2) \ef{\uv^n-\uv^{n+1}}:\ef{\uv^n-\uv^\star} \dx.
\end{align*}
Let us recall that, for $a,b \geq 0$, $ab \leq \frac{1}{2\gamma}a^2+\frac{\gamma}{2}b^2$ for all $\gamma >0$: Indeed, this holds true as the function $\gamma \mapsto \frac{1}{2\gamma}a^2+\frac{\gamma}{2}b^2$ takes its minimum $ab$ at $\gamma=\nicefrac{a}{b}$. Consequently, we obtain 
\begin{align} \label{eq:thmsummands}
\E(\uv^n)-\E(\uv^\star) \leq \frac{1}{2\gamma} \int_\Omega \mu(|\ef{\uv^n}|^2)|\ef{\uv^n-\uv^{n+1}}|^2 \dx + \frac{\gamma}{2} \int_\Omega \mu(|\ef{\uv^n}|^2)|\ef{\uv^n-\uv^{\star}}|^2 \dx.
\end{align}
We will now examine the two summands on the right-hand side above separately.

The first summand can be bounded from above in a similar manner as in the proof of~\cite[Thm.~18]{Diening:2020}. First, we note that
\begin{align*}
\int_\Omega \mu(|\ef{\uv^n}|^2)|\ef{\uv^n-\uv^{n+1}}|^2 \dx &= \int_\Omega \mu(|\ef{\uv^n}|^2)|\ef{\uv^n}|^2 \dx- \int_\Omega\mu(|\ef{\uv^n}|^2)|\ef{\uv^{n+1}}|^2 \dx \\
& \quad -2 \int_\Omega \mu(|\ef{\uv^n}|^2)\ef{\uv^{n+1}}:\ef{\uv^{n}} \dx \\
& \quad +2 \int_\Omega\mu(|\ef{\uv^n}|^2)\ef{\uv^{n+1}}:\ef{\uv^{n+1}} \dx,
\end{align*}
and thus, by the definition of the operator $\A$, cf.~\eqref{eq:Aoperator},
\begin{align*}
\int_\Omega \mu(|\ef{\uv^n}|^2)|\ef{\uv^n-\uv^{n+1}}|^2 \dx &= \A[\uv^n](\uv^n)(\uv^n)-\A[\uv^n](\uv^{n+1})(\uv^{n+1}) \\
&\quad -2 \A[\uv^n](\uv^{n+1})(\uv^n)+2 \A[\uv^n](\uv^{n+1})(\uv^{n+1}).
\end{align*}
Recall that $\A[\uv^n](\uv^{n+1})(\vv)=\ell(\vv)$ for all $\vv \in V$, cf.~\eqref{eq:kacanov}, which leads to
\begin{align*} 
\int_\Omega \mu(|\ef{\uv^n}|^2)|\ef{\uv^n-\uv^{n+1}}|^2 \dx &= \A[\uv^n](\uv^n)(\uv^n)-\A[\uv^n](\uv^{n+1})(\uv^{n+1})-2 \ell(\uv^n)+2 \ell(\uv^{n+1}),
\end{align*}
and hence, by~\eqref{eq:keyin},  
\begin{align} \label{eq:thmfirstsummand}
\frac{1}{2}\int_\Omega \mu(|\ef{\uv^n}|^2)|\ef{\uv^n-\uv^{n+1}}|^2 \dx\leq \E(\uv^n)-\E(\uv^{n+1}).
\end{align}

Next, we will take care of the second summand in~\eqref{eq:thmsummands}. As was done in~\cite{Diening:2020}, we want to bound this summand in terms of the energy difference $\E(\uv^n)-\E(\uv^\star)$. However, in order to improve the contraction factor whilst removing all unknown constants, some modifications to the argument presented in~\cite{Diening:2020} are necessary. For $\psi(t)=\E(\uv^\star+t(\uv^n-\uv^\star))$, the fundamental theorem of calculus implies that
\begin{align*}
\E(\uv^n)-\E(\uv^\star) = \int_0^1 \psi'(t) \, \mathrm{d} t = \int_0^1 \E'(\uv^\star+t(\uv^n-\uv^\star))(\uv^n-\uv^\star) \, \mathrm{d} t.
\end{align*} 
Recall that $\E'(\uv)(\vv)=\A[\uv][\uv](\vv)-\ell(\vv)$, cf.~\eqref{eq:Efrechet2}, and, since $\uv^\star \in V$ is the unique solution of~\eqref{eq:weakoperatorform}, $\ell(\vv)=\A[\uv^\star](\uv^\star)(\vv)$ for any $\vv \in V$. As a consequence, we have that
\begin{align*}
\E(\uv^n)-\E(\uv^\star) &= \int_0^1 \left(\A[\uv^\star+t(\uv^n-\uv^\star)](\uv^\star+t(\uv^n-\uv^\star))-\A[\uv^\star](\uv^\star)\right)(\uv^n-\uv^\star) \, \mathrm{d} t. 
\end{align*}
Invoking the definition of $\A$, cf.~\eqref{eq:Aoperator}, and~\eqref{eq:mukappal} further implies that
\begin{align} \label{eq:eq1}
\E(\uv^n)-\E(\uv^\star) &\geq \int_0^1 t \int_\Omega c(\ef{\uv^\star}+t(\ef{\uv^n}-\ef{\uv^\star}),\ef{\uv^\star}) |\ef{\uv^n-\uv^\star}|^2 \dx \, \mathrm{d}t. 
\end{align}
Moreover, by the definition of $c(\cdot,\cdot)$,~cf.~\eqref{eq:mukappaconstl}, and a brief argument based on reflection we get
\begin{align} \label{eq:cin}
c(\ef{\uv^\star}+s(\ef{\uv^n}-\ef{\uv^\star}),\ef{\uv^\star}) \geq \inf_{t \in (-1,1)} \xi'(\max\left\{0,|\ef{\uv^\star}|+t|\ef{\uv^n}-\ef{\uv^\star}|\right\})
\end{align}
for all $s \in [0,1]$, where $\xi(t)=\mu(t^2)t$; indeed, it is easily verified that, for any $\underline{\kappa},\underline{\tau} \in \mathbb{R}^{d\times d}$, we have $c(\ukappa,\utau)=c(\utau,\ukappa)$, and, in turn, 
\begin{align} \label{eq:eq2}
c(\ukappa+s(\utau-\ukappa),\ukappa)=\inf_{t \in (0,1)}\xi'((1-t) |\ukappa|+t|\ukappa+s(\utau-\ukappa)|), \qquad s \in [0,1].
\end{align} 
The triangle inequality yields that 
\[|\ukappa|-ts|\utau-\ukappa| \leq (1-t)|\ukappa|+t|\ukappa+s(\utau-\ukappa)| \leq |\ukappa|+ts|\utau-\ukappa| \qquad \text{for all} \ t \in (0,1),\] 
and thus
\[
(1-t)|\ukappa|+t|\ukappa+s(\utau-\ukappa)|=|\ukappa|+(2r-1)st|\ukappa-\utau| \qquad \text{for some} \ r \in [0,1].\] This further implies that, for any $s \in [0,1]$, we have
\[
\{(1-t)|\ukappa|+t|\ukappa+s(\utau-\ukappa)|: t \in (0,1)\} \subseteq \{|\ukappa|+t |\ukappa - \utau|:t \in (-1,1)\},
\]
and consequently, in regard of~\eqref{eq:eq2}, $c(\ukappa+s(\utau-\ukappa),\ukappa) \geq \inf_{t \in (-1,1)} \xi'(\max\left\{0,|\ukappa|+ts|\utau-\ukappa|\right\})$, which immediately implies~\eqref{eq:cin}. Combining the equalities~\eqref{eq:eq1} and~\eqref{eq:cin} yields
\begin{align} \label{eq:difinf}
\E(\uv^n)-\E(\uv^\star) \geq \frac{1}{2} \int_\Omega \inf_{t \in (-1,1)} \xi'(\max\left\{0,|\ef{\uv^\star}|+t|\ef{\uv^n}-\ef{\uv^\star}|\right\}) |\ef{\uv^n-\uv^\star}|^2 \dx.
\end{align}
Since, in addition,
\begin{align*} 
\frac{1}{2} \int_\Omega \mu(|\ef{\uv^n}|^2)|\ef{\uv^n-\uv^{\star}}|^2 \dx &= \frac{1}{2} \int_\Omega \frac{\mu(|\ef{\uv^n}|^2)}{\inf_{t \in (-1,1)} \xi'(\max\left\{0,|\ef{\uv^\star}|+t|\ef{\uv^n}-\ef{\uv^\star}|\right\})} \\
&\quad \cdot \inf_{t \in (-1,1)} \xi'(\max\left\{0,|\ef{\uv^\star}|+t|\ef{\uv^n}-\ef{\uv^\star}|\right\})|\ef{\uv^n-\uv^\star}|^2 \dx,
\end{align*}
the lower bound~\eqref{eq:difinf} implies that
\begin{align} \label{eq:summandtwo}
\frac{1}{2} \int_\Omega \mu(|\ef{\uv^n}|^2)|\ef{\uv^n-\uv^{\star}}|^2 \dx \leq Q(n) (\E(\uv^n)-\E(\uv^\star)),
\end{align}
where 
\begin{align*}
Q(n):=\esssup_{\x \in \Omega} \frac{\mu(|\ef{\uv^n}|^2)}{\inf_{t \in (-1,1)} \xi'(\max\left\{0,|\ef{\uv^\star}|+t|\ef{\uv^n}-\ef{\uv^\star}|\right\})}.
\end{align*}
Finally, combining~\eqref{eq:thmsummands}, \eqref{eq:thmfirstsummand}, and~\eqref{eq:summandtwo} yields
\begin{align*}
\gamma (1-\gamma Q(n)) \left(\E(\uv^n)-\E(\uv^\star)\right) \leq \E(\uv^n)-\E(\uv^{n+1}), 
\end{align*}
and, in turn,
\begin{align*}
\E(\uv^{n+1})-\E(\uv^\star)
&=\E(\uv^n)-\E(\uv^\star)-\left(\E(\uv^n)-\E(\uv^{n+1})\right)  \leq (1-\gamma(1-\gamma Q(n)))(\E(\uv^n)-\E(\uv^\star)). 
\end{align*}
It is straightforward to verify that the contraction factor is minimal for $\gamma=\nicefrac{1}{2Q(n)}$, and, in that case, one has that
\begin{align*}
0 \leq \E(\uv^{n+1})-\E(\uv^\star) \leq \left(1-\frac{1}{4 Q(n)}\right) \left(\E(\uv^n)-\E(\uv^\star)\right)\!,
\end{align*}
which proves the claim.
\end{proof}

\begin{remark} \label{rem:boundconfac}
Since $m_\mu \leq \mu(t) \leq M_\mu$ as well as $m_\mu \leq \xi'(t) \leq M_\mu$ for all $t \geq 0$, we get 
the following rough uniform bound on the contraction factor: 
\[q(n) \leq \left(1-\frac{m_\mu}{4 M_\mu}\right) \in [0.75,1), \qquad n \geq 0.\]
We note that, in the context of the relaxed power-law model, cf. \S\ref{sec:powerlaw}, this bound, in principle, coincides with the contraction factor from~\cite{Diening:2020}. 
\end{remark}

We note that the contraction factor~\eqref{eq:contractionfactor} is not computable as it involves $\uv^\star$, and the uniform upper bound from Remark~\ref{rem:boundconfac} is rather pessimistic. In the following, we will establish an improved computable bound, up to higher order error terms, for the contraction factor on finite-dimensional subspaces.

\begin{theorem} \label{thm:globalcontractioncomp}
Assume that {\rm (A1)--(A3)} hold, let $W \subset V$ be a finite-dimensional subspace, and let $\E$ be defined as in~\eqref{eq:Edeff} (restricted to $W$). Then, the energy error contracts along the sequence $\{\uv^n\} \subset W$  generated by the Ka\v{c}anov iteration~\eqref{eq:kacanov} on $W$ in the sense that 
\begin{align*} 
0\leq\E(\uv^{n+1})-\E(\uv^\star) \leq q_A(n) \left(\E(\uv^n)-\E(\uv^\star)\right)+o_W(\nnn{\uv^\star-\uv^n}_\Omega^2),
\end{align*}
where now $\uv^\star$ denotes the unique minimiser of $\E$ in $W$ and
\begin{align} \label{eq:contractionfactorcomp}
q_A(n):=1-\frac{1}{4} \left\{\esssup_{\x \in \Omega} \frac{\mu(|\ef{\uv^n}|^2)}{2 \mu'(|\ef{\uv^n}|^2) |\ef{\uv^n}|^2+\mu(|\ef{\uv^n}|^2)}\right\}^{-1}.
\end{align}
\end{theorem}

\begin{proof}
This result follows, in principle, from the proof of Theorem~\ref{thm:globalcontraction} with a modification of the bound from~\eqref{eq:summandtwo}.
Consider the map $\omega:W \to W^\star$ given by $\omega(\uv):=\A[\uv](\uv)$ for $\uv \in W$. A lengthy, but straightforward calculation reveals that the G\^{a}teaux derivative of $\omega$ exists and is given by
\begin{align*}
\omega'(\uv)(\vv)(\wv)=\int_\Omega 2 \mu'(|\eu|^2)(\eu:\ev)(\eu:\e{\wv})+\mu(|\eu|^2)(\ev:\e{\wv}) \dx, \quad \vv,\wv \in W.
\end{align*}
Since $W$ is a finite-dimensional space and $\omega:W \to W^\star$ is Lipschitz continuous by Proposition~\ref{lem:properties}, the G\^{a}teaux  derivative coincides with the Fr\'{e}chet derivative, see~\cite[Prop.~3.5]{Shapiro:1990}. By definition of the Fr\'{e}chet derivative, one has that
\begin{align*}
\omega(\uv)=\omega(\uv^n)+\omega'(\uv^n)(\uv-\uv^n)+o_W(\nnn{\uv-\uv^n}_\Omega);
\end{align*}
here, $o_W(\nnn{\uv-\uv^n}_\Omega)$ denotes a remainder in the dual space $W^\star$. Combining these two observations yields
\begin{multline*}
\left(\A[\uv^\star+t(\uv^n-\uv^\star)](\uv^\star+t(\uv^n-\uv^\star))-\A[\uv^\star](\uv^\star)\right)(\uv^n-\uv^\star) \\
\begin{aligned}
&=\left(\omega(\uv^\star+t(\uv^n-\uv^\star))-\omega(\uv^\star)\right)(\uv^n-\uv^\star)\\
&=t\omega'(\uv^n)(\uv^n-\uv^\star)(\uv^n-\uv^\star)+o_W(\nnn{\uv^\star-\uv^n}_\Omega^2) \\
&= t\int_\Omega 2 \mu'(|\ef{\uv^n}|^2)(\ef{\uv^n}:\ef{\uv^n-\uv^\star})^2 \\
& \quad +\mu(|\ef{\uv^n}|^2)|\ef{\uv^n-\uv^\star}|^2 \dx+o_W(\nnn{\uv^\star-\uv^n}_\Omega^2).
\end{aligned}
\end{multline*}
Recall that, by assumption (A3), $\mu'(t) \leq 0$ for all $t \geq 0$. Therefore, the Cauchy--Schwarz inequality implies that 
\begin{multline*}
\left(\A[\uv^\star+t(\uv^n-\uv^\star)](\uv^\star+t(\uv^n-\uv^\star))-\A[\uv^\star](\uv^\star)\right)(\uv^n-\uv^\star) \\ 
\begin{aligned}
&\geq t \int_\Omega \left\{2 \mu'(|\ef{\uv^n}|^2) |\ef{\uv^n}|^2+\mu(|\ef{\uv^n}|^2)\right\}|\underline{e}(\uv^n-\uv^\star)|^2 \dx\\
& \quad +o_W(\nnn{\uv^\star-\uv^n}_\Omega^2).
\end{aligned}
\end{multline*}
Consequently,
\begin{align*}
\E(\uv^n)-\E(\uv^\star) &= \int_0^1 \left(\A[\uv^\star+t(\uv^n-\uv^\star)](\uv^\star+t(\uv^n-\uv^\star))-\A[\uv^\star](\uv^\star)\right)(\uv^n-\uv^\star) \, \mathrm{d} t \\
& \geq \frac{1}{2} \int_\Omega \left\{2 \mu'(|\ef{\uv^n}|^2) |\ef{\uv^n}|^2+\mu(|\ef{\uv^n}|^2)\right\}|\underline{e}(\uv^n-\uv^\star)|^2 \dx +o_W(\nnn{\uv^\star-\uv^n}_\Omega^2),
\end{align*}
and thus
\begin{align*}
\frac{1}{2} \int_\Omega \mu(|\ef{\uv^n}|^2)|\ef{\uv^n-\uv^\star}|^2 \dx &\leq \esssup_{\x \in \Omega} \frac{\mu(|\ef{\uv^n}|^2)}{2 \mu'(|\ef{\uv^n}|^2) |\ef{\uv^n}|^2+\mu(|\ef{\uv^n}|^2)}\Big ((\E(\uv^n)-\E(\uv^\star))\\
&\quad +o_W(\nnn{\uv^\star-\uv^n}_\Omega^2)\Big).
\end{align*} 
The rest follows as in the proof of Theorem~\ref{thm:globalcontraction}. We note, however, that the factor of the remainder term $o_W(\nnn{\uv^\star-\uv^n}_\Omega^2)$ above cancels by the multiplication with $\gamma$ in~\eqref{eq:thmsummands}. 
\end{proof} 

\begin{remark}
We emphasize that the contraction factor $q_A$ from~\eqref{eq:contractionfactorcomp} is independent of the finite-dimensional subspace $W \subset V$. However, the remainder term $o_W$ may depend on the choice of the given discrete subspace, as indicated by the subscript. 
\end{remark}

Finally we remark that the energy error is equivalent to the norm error, i.e., the norm error contracts, up to some uniform constant, along the sequence generated by the Ka\v{c}anov scheme as well. This equivalence was already established in a similar setting, e.g., in~\cite[Lem.~2.3]{HeidWihler2:19v1} and \cite[Lem.~5.1]{GantnerHaberlPraetoriusStiftner:17}. The proof can also be found in those references.

\begin{proposition} \label{lem:energydifference}
Let $\E$ be defined as in \eqref{eq:Edeff}, with $\mu$ satisfying {\rm (A1)--(A3)}, and let $\uv^\star$ be the unique minimiser of
$\E$ in $V$; then,
\begin{equation} \label{eq:energydifference}
  \frac{m_\mu}{2}\nnn{\uv^\star - \vv}_\Omega^2 \leq \E(\vv)-\E(\uv^\star) \leq \frac{\sqrt{3}M_\mu}{2} \nnn{\uv^\star-\vv}^2_\Omega \qquad \text{for all } \vv \in V.
 \end{equation}
An analogous result holds on any finite-dimensional subspace $W \subset V$, with $V$ replaced by $W$ in the assertion above.
\end{proposition}


\subsection{Application to the Carreau model} \label{sec:carreau}

A widely used model for the flow of incompressible non-Newtonian fluids is the Carreau law, cf.~\cite{Carreau:72}. In that case the viscosity coefficient $\mu$ in~\eqref{eq:strong} is of the form
\begin{align} \label{eq:carreau}
\mu(t)=\mu_\infty + (\mu_0-\mu_\infty)(1+ \lambda t)^{\frac{r-2}{2}},
\end{align}
where, for shear-thinning fluids, $r \in (1,2)$, $\lambda > 0$ is the relaxation time, and $0<\mu_\infty<\mu_0<\infty$ denote the infinite and zero shear rate, respectively. The function $\mu$ from~\eqref{eq:carreau} is smooth, decreasing since $r \in (1,2)$, and satisfies the structural assumption (A2), cf.~\eqref{eq:a2}, thanks to the following lemma. 

\begin{lemma} 
Let $r \in (1,2)$, $\lambda >0$, and $0<\mu_\infty < \mu_0 < \infty$. Then, the following inequalities hold:
\begin{align*} 
\mu_\infty (t-s) \leq \mu(t^2)t-\mu(s^2)s \leq \mu_0 (t-s), \qquad t \geq s \geq 0.
\end{align*}
\end{lemma} 

\begin{proof}
Define $\xi(t):=\mu(t^2)t$, $t \geq 0$. The mean value theorem yields
\begin{align*}
\inf_{\tau \geq 0} \xi'(\tau) (t-s) \leq \xi(t)-\xi(s) \leq \sup_{\tau \geq 0} \xi'(\tau)(t-s),
\end{align*}
and thus we need to show that $\mu_\infty=\inf_{\tau \geq 0} \xi'(\tau)$ and $\mu_0=\sup_{\tau \geq 0} \xi'(\tau)$. A straightforward calculation reveals that $\xi''(\tau) \neq 0$ for all $\tau \geq 0$, i.e., $\xi'$ has no local extrema in the interval $(0,\infty)$. Since, in addition, $\lim_{\tau \to 0} \xi'(\tau)=\mu_0$ and $\lim_{\tau \to \infty} \xi'(\tau)=\mu_\infty$, the lemma is established.
\end{proof}

In particular, we may apply Theorem~\ref{thm:globalcontraction} and Theorem~\ref{thm:globalcontractioncomp} to the Carreau model. In this case, the computable contraction factor from~\eqref{eq:contractionfactorcomp} reads as follows, with $\uv^n \in W$:
\begin{align*} 
q_A(n)&:=1-\frac{1}{4}\left(\esssup_{\x \in \Omega}\frac{\mu_\infty+(\mu_0-\mu_\infty)(1+\lambda |\ef{\uv^n}|^2)^{\frac{r-2}{2}}}{\mu_\infty+(\mu_0-\mu_\infty)(1+\lambda |\ef{\uv^n}|^2)^{-1+\frac{r-2}{2}}(1+\lambda(r-1)|\ef{\uv^n}|^2)}\right)^{-1} \\
&=1-\frac{1}{4}\essinf_{\x \in \Omega}\frac{\mu_\infty+(\mu_0-\mu_\infty)(1+\lambda |\ef{\uv^n}|^2)^{-1+\frac{r-2}{2}}(1+\lambda(r-1)|\ef{\uv^n}|^2)}{\mu_\infty+(\mu_0-\mu_\infty)(1+\lambda |\ef{\uv^n}|^2)^{\frac{r-2}{2}}}. 
\end{align*}
Let us further examine this factor. First we note that
\[\frac{\mu_\infty+(\mu_0-\mu_\infty)(1+\lambda t^2)^{-1+\frac{r-2}{2}}(1+\lambda(r-1)t^2)}{\mu_\infty+(\mu_0-\mu_\infty)(1+\lambda t^2)^{\frac{r-2}{2}}} \to 1 \qquad \text{as } t \to \infty,\]  
which is optimal from the point of view if contraction. Consequently, we do not expect a significant deterioration of the convergence rate if the the rate-of-strain tensor of the solution, i.e., $\ef{\uv^\star}$, is unbounded, cf.~Experiment~\ref{exp:errordecaycarreau}.

Moreover, an elementary calculation reveals that 
\[ \frac{\mu_\infty+(\mu_0-\mu_\infty)(1+\lambda t^2)^{-1+\frac{r-2}{2}}(1+\lambda(r-1)t^2)}{\mu_\infty+(\mu_0-\mu_\infty)(1+\lambda t^2)^{\frac{r-2}{2}}} \geq \frac{1+\lambda(r-1)t^2}{1+\lambda t^2} \qquad \text{for all } t \geq 0,\] 
and therefore
\begin{align*}
q_A(n) &=1-\frac{1}{4}\essinf_{\x \in \Omega}\frac{\mu_\infty+(\mu_0-\mu_\infty)(1+\lambda |\ef{\uv^n}|^2)^{-1+\frac{r-2}{2}}(1+\lambda(r-1)|\ef{\uv^n}|^2)}{\mu_\infty+(\mu_0-\mu_\infty)(1+\lambda |\ef{\uv^n}|^2)^{\frac{r-2}{2}}}\\
& \leq 1-\frac{1}{4} \essinf_{\x \in \Omega} \frac{1+\lambda (r-1) |\ef{\uv^n}|^2}{1+ \lambda|\ef{\uv^n}|^2}\\& \leq 1-\frac{1}{4}(r-1). 
\end{align*}
In combination with Remark~\ref{rem:boundconfac}, we get
\begin{align} \label{eq:qaC}
q_A(n) \leq \min\left\{1-\frac{1}{4} \frac{\mu_\infty}{\mu_0},1-\frac{1}{4}(r-1)\right\}\!,
\end{align}
i.e., the convergence rate may only deteriorate drastically if $r\to 1$ and, in addition, $\nicefrac{\mu_\infty}{\mu_0} \to 0$. 


\subsection{Application to the relaxed power-law model} \label{sec:powerlaw}

Another prominent model for non-New\-tonian fluids, e.g., in polymer processing, is the power-law model, see, e.g.,~\cite[Ch.~3.3]{tadmor:2006}. For this model, the weak formulation~\eqref{eq:directweak} of the boundary-value problem under consideration is as follows:
\begin{align} \label{eq:powerlawmodel}
\text{find } \uv \in X \ \text{such that} \qquad \int_\Omega |\eu|^{r-2} \eu:\ev \dx =\ell(\vv) \qquad \text{for all } \vv \in X;
\end{align}
here, $X:=\{\uv \in \mathrm{W}_0^{1,r}(\Omega)^d: \nabla \cdot \uv=0\}$ and $\ell \in X^\star$, where, for shear-thinning fluids, $r \in (1,2)$. In particular, the viscosity coefficient is given by
\begin{align*}
\mu(t)=t^{\frac{r-2}{2}}.
\end{align*}
Clearly, $\mu:\mathbb{R}_{> 0} \to \mathbb{R}_{> 0}$ is neither bounded away from zero nor bounded from above, i.e., (A2) is not satisfied. Therefore, as was proposed in the work~\cite{Diening:2020}, we consider a relaxed version of $\mu$: for $0<\varepsilon_{-}<\varepsilon_{+}<\infty$ we define the viscosity coefficient
\begin{align} \label{eq:relaxedpowerlaw}
\mu_{\varepsilon}(t):=
\begin{cases}
\varepsilon_{-}^{r-2} &\quad \, \text{for}\ 0 \leq t < \varepsilon_{-}^2, \\
t^{\frac{r-2}{2}} &\quad \,  \text{for}\ \varepsilon_{-}^2 \leq t \leq \varepsilon_{+}^2, \\
\varepsilon_{+}^{r-2} &\quad \, \text{for}\ t \geq \varepsilon_{+}^2.
\end{cases}
\end{align} 
The function $\mu_\varepsilon$ is decreasing, strictly positive, bounded, globally Lipschitz continuous, and satisfies (A2) with
\[(r-1)\varepsilon_{+}^{r-2}(t-s) \leq \mu(t^2)t-\mu(s^2)s \leq \varepsilon_{-}^{r-2}(t-s), \qquad t \geq s \geq 0;\] 
it is, furthermore, differentiable at all $t \in [0,\infty)\setminus\{\varepsilon_{-}^2,\varepsilon_{+}^2\}$ and has finite left- and at right-derivatives at $t=\varepsilon_{-}^2$ and $t=\varepsilon_{+}^2$. Hence, even though $\mu_{\varepsilon}$ is not continuously differentiable on $[0,\infty)$, Theorem~\ref{thm:globalcontraction} can be, nevertheless, applied in the given setting. Moreover, in the generic case when the set $\Omega_S^n:=\{\x \in \Omega: |\ef{\uv^n(\x)}| \in \{\varepsilon_{-},\varepsilon_{+}\}\}$, for every $n \geq 0$, has Lebesgue measure zero, the operator $\omega$ from the proof of Theorem~\ref{thm:globalcontractioncomp} is Fr\'{e}chet differentiable at $\uv^n \in W$. Thus, in turn, Theorem~\ref{thm:globalcontractioncomp} can then also be employed to the relaxed power-law model\footnote{Nonetheless we will present a continuously differentiable version of the viscosity coefficient~\eqref{eq:relaxedpowerlaw} in the Appendix~\ref{app:differentiablerelaxatio}.}. A simple calculation reveals that the computable contraction factor from~\eqref{eq:contractionfactorcomp} can again be bounded; indeed,
\begin{align} \label{eq:facpl}
q_A(n) \leq 1-\frac{1}{4}(r-1).
\end{align}
Moreover, one even has that $q_A(n)=1-4^{-1}(r-1)$ if the set $\{\x \in \Omega: \varepsilon_{-} \leq |\ef{\uv^n(\x)}| \leq \varepsilon_{+}\}$ is of positive Lebesgue measure. We further remark that 
\[
\frac{m_\mu}{M_\mu}=\frac{(r-1) \varepsilon_{+}^{r-2}}{\varepsilon_{-}^{r-2}} < (r-1),
\]
since $r \in (1,2)$. This shows that the bound~\eqref{eq:facpl} is, for every value $r \in (1,2)$, sharper than the bound from Remark~\ref{rem:boundconfac}. Furthermore, this bound predicts that it is the physical parameter $r$ that affects the convergence rate of the iteration, in the finite-dimensional setting at least, rather than the quotient $\nicefrac{\varepsilon_{+}^{r-2}}{\varepsilon_{-}^{r-2}}$ implied by existing bounds on the contraction factor, cf.~\cite[Cor.~19]{Diening:2020}. Significantly, the upper bound $(r-1)$ on the contraction factor appearing of the right-hand side of~\eqref{eq:facpl} is independent of the relaxation parameters $\varepsilon_{\pm}$. This is of importance as we are interested in the power-law model~\eqref{eq:powerlawmodel} and we thus need to let $\varepsilon_{-} \to 0$ and $\varepsilon_{+} \to \infty$. We note that the existence of a bound independent of $\varepsilon_{\pm}$ on the contraction factor of the relaxed Ka\v{c}anov iteration applied to the power-law model with $r \in (1,2)$ was stated in the infinite-dimensional case as an open problem in~\cite[Ex.~20]{Diening:2020}.

We further note that the energy functional $\E_{\varepsilon}$ corresponding to the viscosity from~\eqref{eq:relaxedpowerlaw} coincides with the energy functional $\mathcal{J}_\epsilon$ from~\cite{Diening:2020} up to a constant shift depending on $\varepsilon_{-}$. To be precise, one has that
\begin{align*}
\E_{\varepsilon}(\uv)=\mathcal{J}_\epsilon(\uv)+\left(\frac{1}{2}-\frac{1}{r}\right)\varepsilon_{-}^r, \qquad \uv \in V,
\end{align*}
and thus the results established in~\cite{Diening:2020} may be directly applied in our setting. In particular, this implies that the unique minimiser $\uv^\star_\varepsilon \in V$ of $\E_{\varepsilon}$ converges in $\mathrm{W}^{1,r}_{0}(\Omega)$ to the unique minimiser $\uv^\star \in X$ of 
\[\E(\uv)=\frac{1}{r}\int_\Omega |\eu|^r\dx-\ell(\uv),\]
cf.~\cite[Cor.~10]{Diening:2020}. 

\begin{remark}
The relaxed power-law model could also be solved by using a (damped) Newton method, cf.~\cite[Prop.~5.3]{HeidWihler:19v2}. However, it is unclear whether and how the convergence rate will deteriorate as $\varepsilon_{-} \to 0$ and $\varepsilon_{+} \to \infty$. For an application of Newton's method to the power-law model with a different regularisation approach we refer to~\cite{Hirn:2013}; however, the convergence rate in relation to the choice of the regularisation parameter $\varepsilon$ is not examined in that work.
\end{remark}

\begin{remark}
We emphasise that our analysis does also apply to a variable (measurable) exponent $r:\Omega \to (1,2)$ for both the relaxed power-law model as well as the Carreau model. Then, in~\eqref{eq:facpl} and~\eqref{eq:qaC}, respectively, we need to replace $1-\nicefrac{1}{4}(r-1)$ by $1-\nicefrac{1}{4}(\essinf_{\x \in \Omega} r(\x)-1)$.
\end{remark}

\section{Experiments} \label{sec:numericalexp}

In this section, we will perform some numerical tests to assess our findings. To this end, we consider the simplified problem
\begin{align*}
\text{find} \ u \in \HS \ \text{such that} \qquad \int_\Omega \mu(|\nabla u|^2)\nabla u \cdot \nabla v= \int_\Omega fv \dx \qquad \text{for all } v \in \HS,
\end{align*}
where $\Omega:=(-1,1)^2 \setminus [0,1]\times[-1,0] \subset \mathbb{R}^2$ is an L-shaped domain, $f \in \mathrm{L}^2(\Omega)$, and the coefficient $\mu$ either obeys the Carreau law~\eqref{eq:carreau} or the relaxed power-law~\eqref{eq:relaxedpowerlaw}. We remark that the theory derived before equally applies to this simpler case. In all our experiments below, we use a conforming P1-finite element discretisation, where the mesh consists of $\mathcal{O}(10^6)$ triangles, except where explicitly stated otherwise.

\subsection{Error decay in dependence on $r$} First, we will examine how the convergence rate of the \emph{error} depends on the exponent $\frac{r-2}{2}$; recall that the norm error is equivalent to the energy error, cf.~Proposition~\ref{lem:energydifference}. This will be done for both the Carreau and the relaxed power-law model, for smooth and irregular solutions. 

\subsubsection{Error decay for the Carreau model} \label{exp:errordecaycarreau}
Let the function $\mu$ obey the Carreau law~\eqref{eq:carreau}, with $\mu_\infty=1$, $\mu_0=100$, $\lambda=2$, and varying values of $r \in (1,2)$. The source term $f$ is chosen so that the unique solution is given by
\begin{enumerate}[(a)]
\item the smooth function $u^\star(x,y)=\sin(\pi x)\sin(\pi y)$, where $(x,y) \in \mathbb{R}^2$ denote the Euclidean coordinates;
\item the function \[u^\star(R,\varphi)=R^{\nicefrac{2}{3}}\sin\left(\nicefrac{2\varphi}{3}\right)(1-R \cos(\varphi))(1+R \cos(\varphi))(1- R \sin(\varphi))(1+R \sin(\varphi))\cos(\varphi),\]
where $R$ and $\varphi$ are polar coordinates, which exhibits a singularity at the origin $(0,0)$.
\end{enumerate}
In the smooth case (a) the mesh is uniform, and in the singular case (b) the mesh is increasingly refined in the vicinity of the singularity point (0,0). In Figure~\ref{fig:CarreauError} we plot the error $\norm{\nabla u^n- \nabla u^\star}_{\mathrm{L}^2(\Omega)}$ against the number of iterative steps $n$. We can clearly see that the convergence rate deteriorates with decreasing $r$, as was predicted in \S\ref{sec:energycontraction}. We further note that the irregularity of the solution in (b) does not affect the convergence rate, as was conjectured in \S\ref{sec:carreau}.
  
\begin{figure} 
{\includegraphics[width=0.48\textwidth]{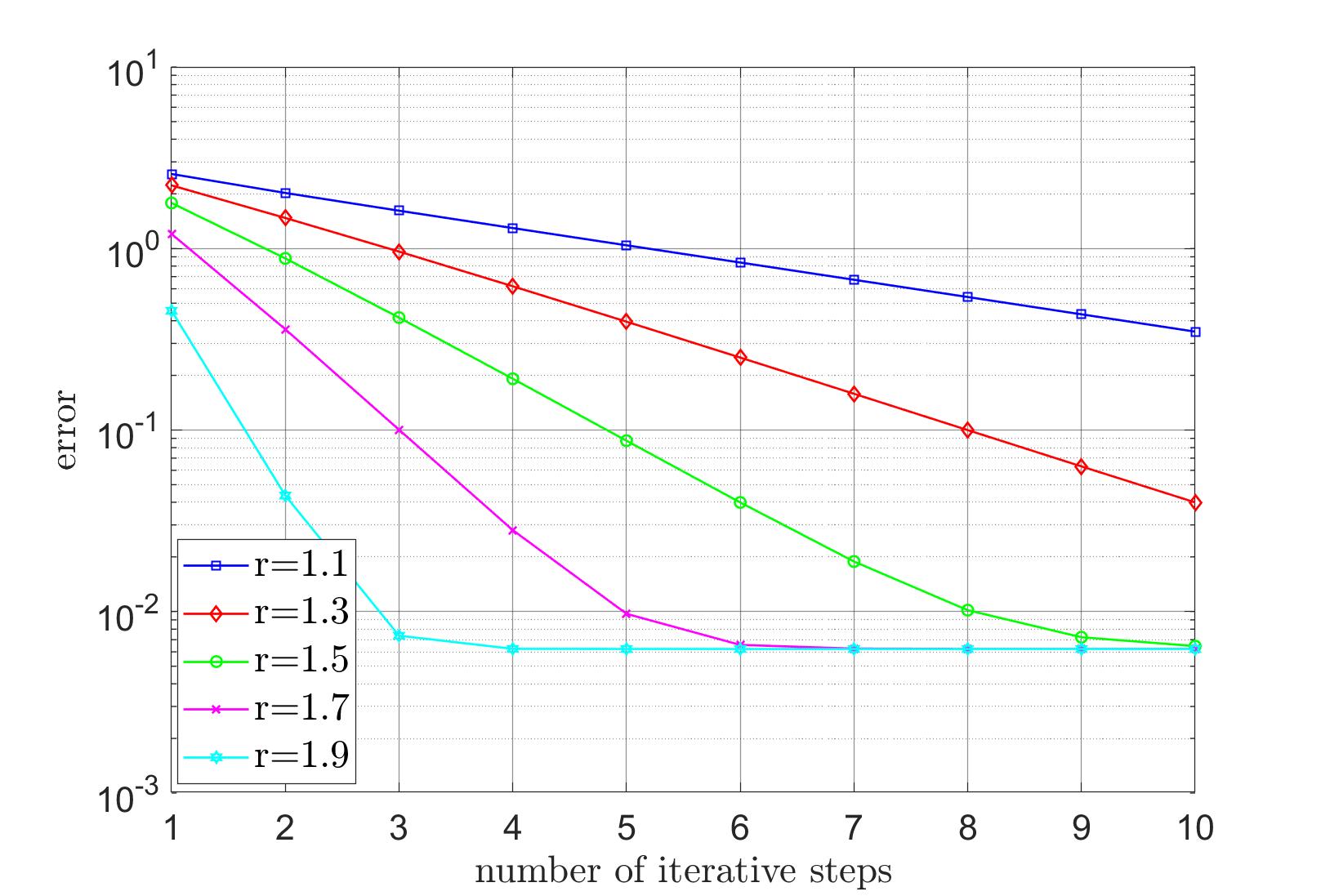}}\hfill
{\includegraphics[width=0.48\textwidth]{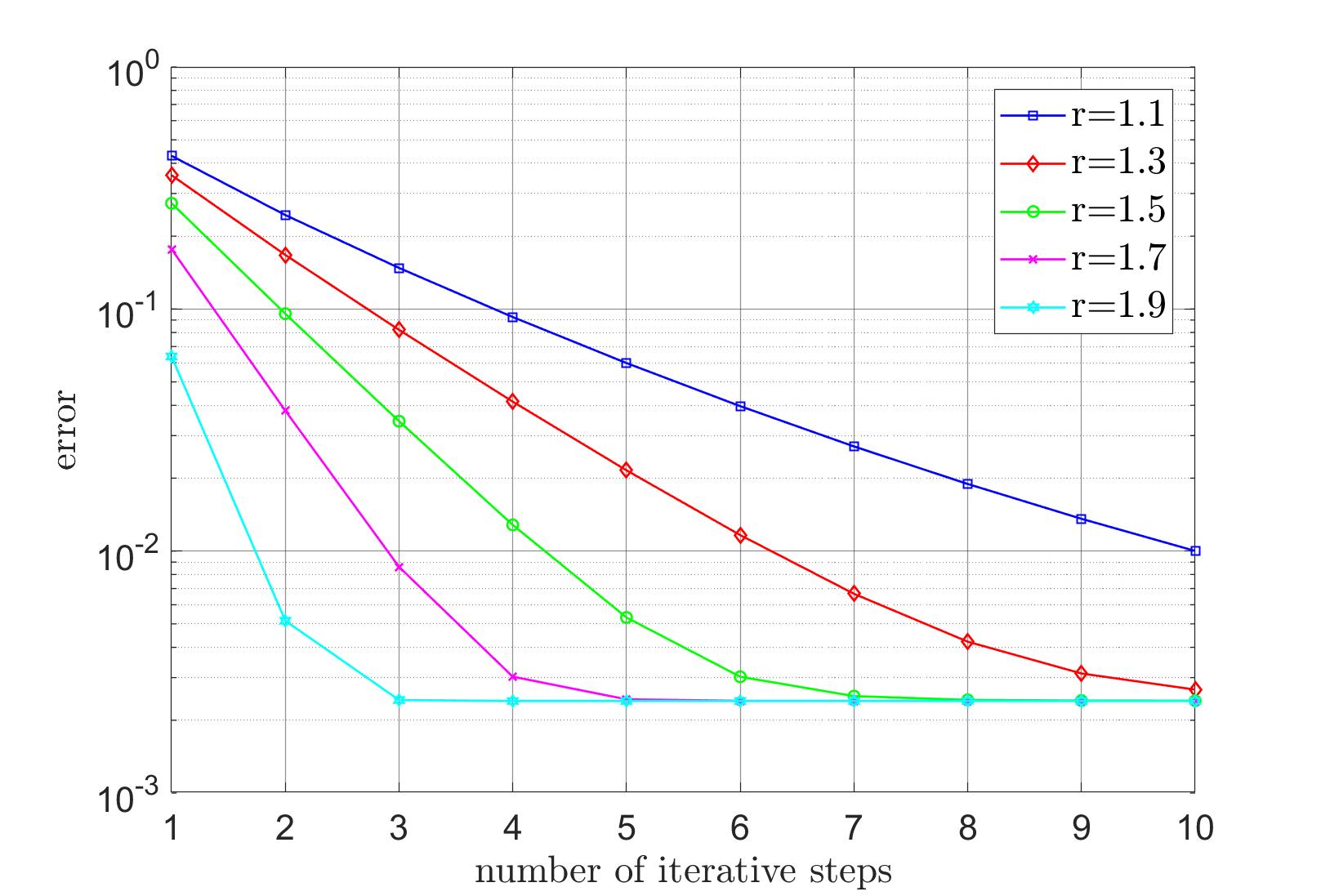}}
 \caption{Carreau model: Influence of the physical parameter $r$ on the convergence rate in the smooth case (left) and irregular case (right), where $\mu_\infty =1$, $\mu_0=100$, and $\lambda=2$.}\label{fig:CarreauError}
\end{figure}

\subsubsection{Error decay for the relaxed power-law model.} Now consider the relaxed power-law model, cf.~\eqref{eq:relaxedpowerlaw}, with $\varepsilon_{-}=10^{-6}$ and $\varepsilon_{+}=10^6$. As before, the source term $f$ is chosen so that (a) $u^\star$ is smooth, and (b) $u^\star$ exhibits a singularity at the origin (0,0). In Figure~\ref{fig:PowerLawError}, the error $\norm{\nabla u^n- \nabla u^\star}_{\mathrm{L}^2(\Omega)}$ is plotted against the number of iterative steps $n$. We observe that for the power-law model the dependence of the convergence rate on the exponent is even stronger than for the Carreau model.

\begin{figure} 
{\includegraphics[width=0.48\textwidth]{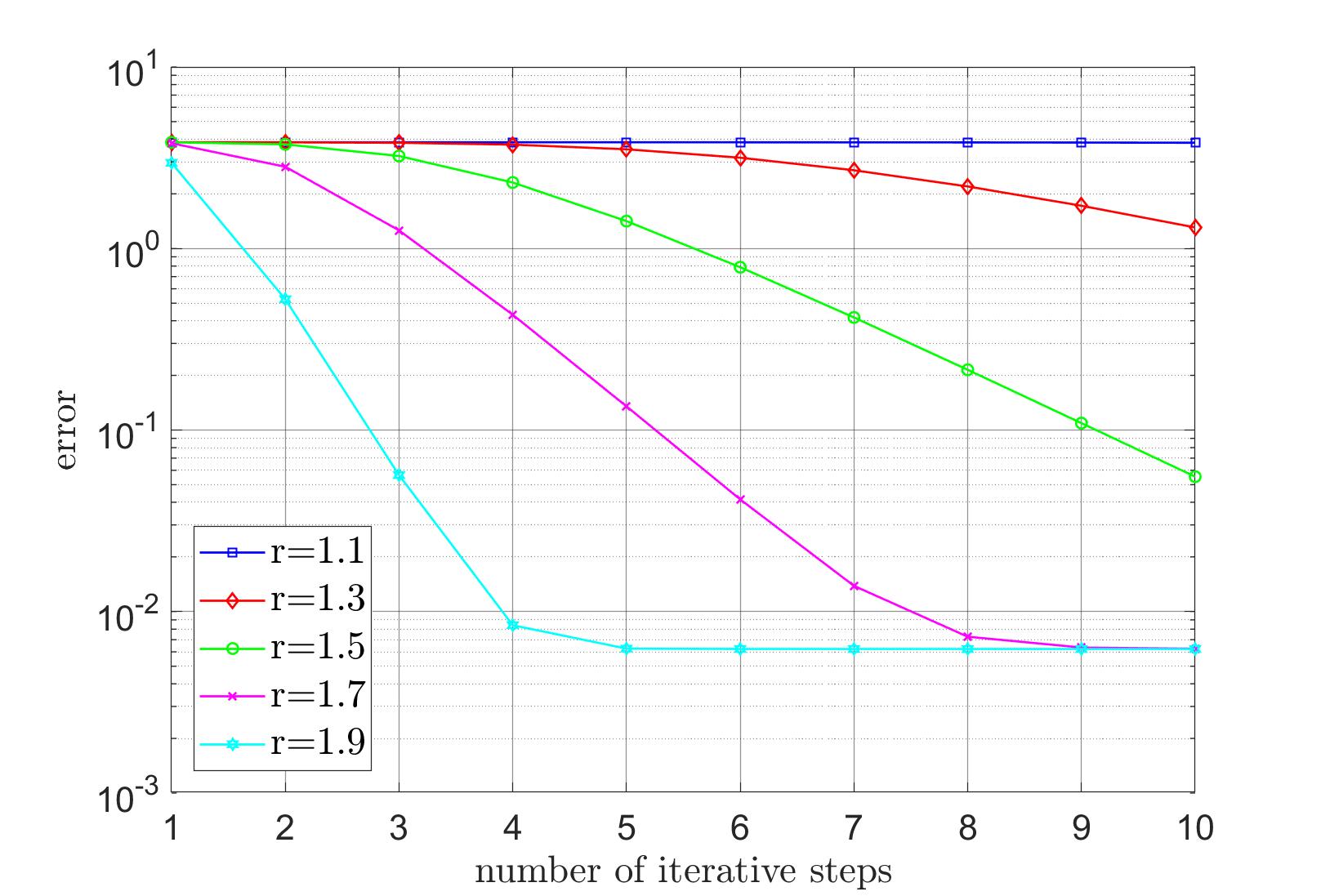}}\hfill
{\includegraphics[width=0.48\textwidth]{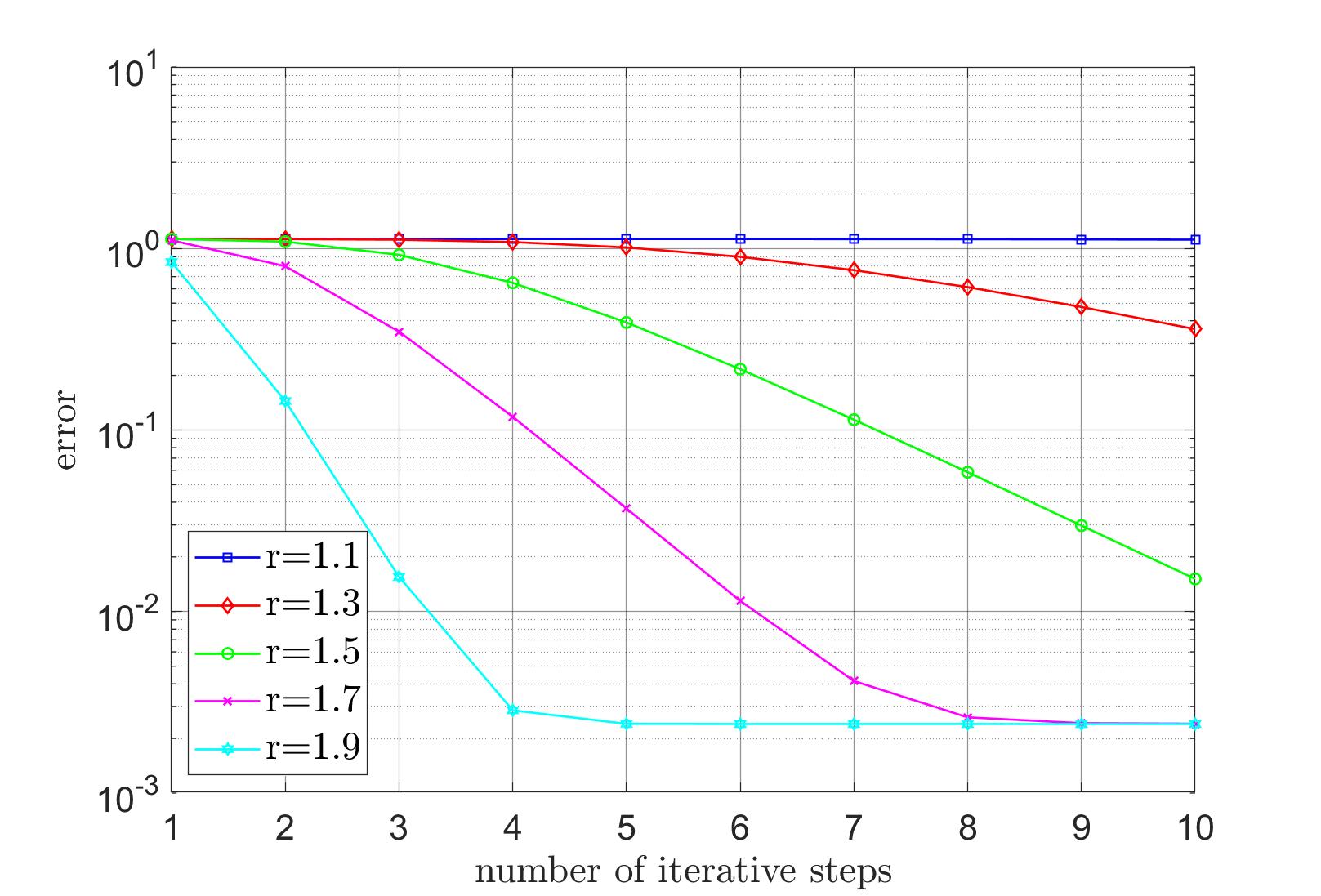}}
 \caption{Relaxed power-law model: Influence of the physical parameter $r$ on the convergence rate in the smooth case (left) and irregular case (right).}\label{fig:PowerLawError}
\end{figure}

\subsubsection{Error decay for close to constant viscosity}
In the experiments before we had that the ratio of the infinite and zero shear rates was much smaller than $(r-1)$, cf.~\eqref{eq:qaC}. Now we choose the parameters so that the `shear stress' depends almost linearly on the `shear rate', and further we let the source term $f$ be such that the unique solution of~\eqref{eq:directweak} is given by the smooth function $u^\star(x,y)=\sin(\pi x) \sin(\pi y)$. For the Carreau law we set $\mu_\infty=1$, $\mu_0=2$, $\lambda=2$, and take again varying values of $r \in (1,2)$; we emphasize that, in this test, we consider even smaller values of $r$ than in the experiments before. In view of the a posteriori computable contraction factor~\eqref{eq:qaC} we expect that the convergence rate will not deteriorate drastically for $r$ close to one, which is confirmed by our numerical experiment, cf.~Figure~\ref{fig:SQ} (left). In the case of the relaxed power-law model, we set $\varepsilon_{-}=1$, $\varepsilon_{+}=2$, and test the same values $r \in (1,2)$ as for the Carreau model. We note that these choices of the relaxation parameters $\varepsilon_{\pm}$ are in practice of no interest, as one is, rather, interested in $\varepsilon_{-} \to 0$ and $\varepsilon_{+} \to \infty$. Nonetheless, we still presume that the convergence deteriorates for $r$ close to one by our analysis in \S\ref{sec:powerlaw}. This is indeed the case, as illustrated in Figure~\ref{fig:SQ} (right).

\begin{figure} 
{\includegraphics[width=0.48\textwidth]{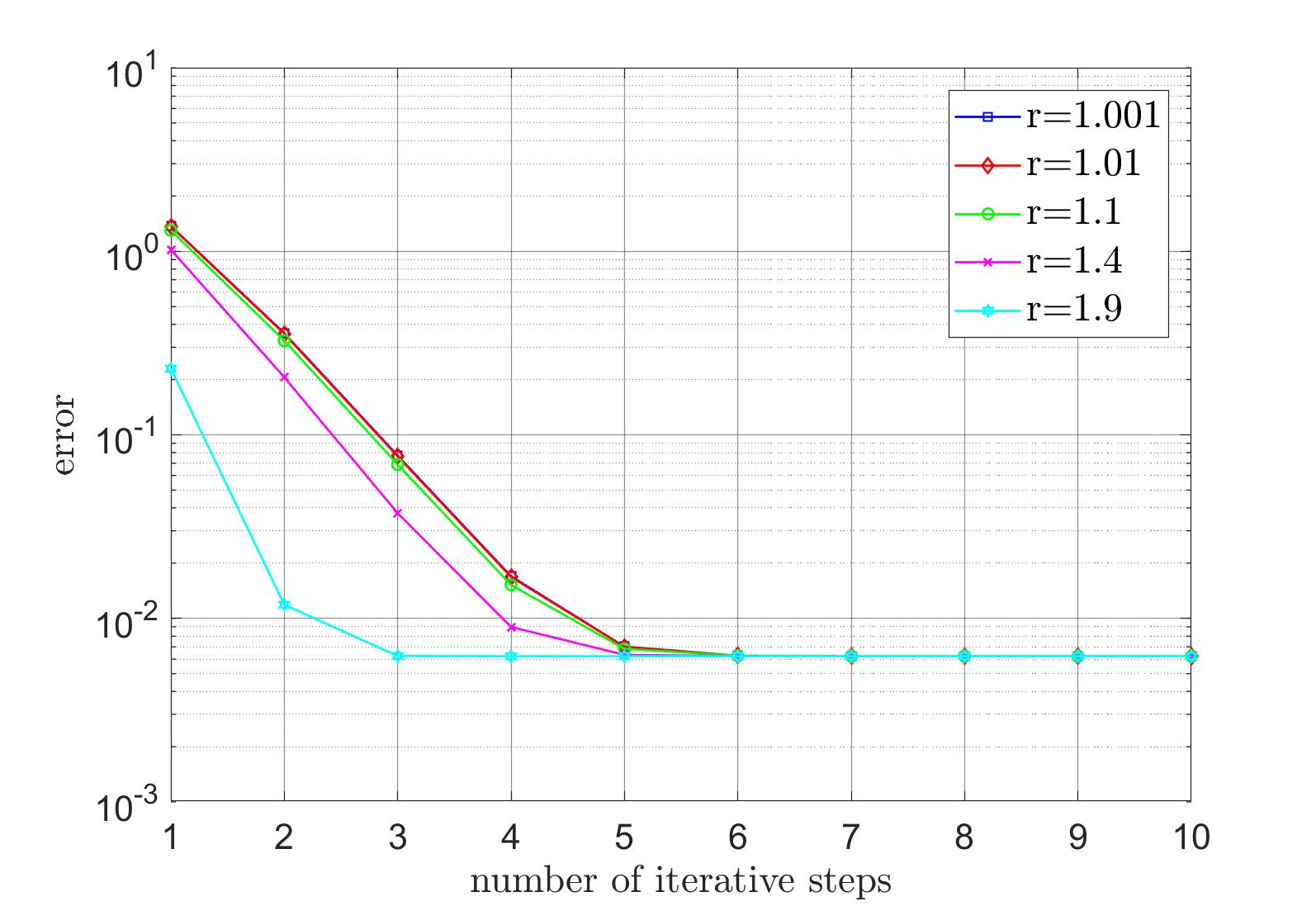}}\hfill
{\includegraphics[width=0.48\textwidth]{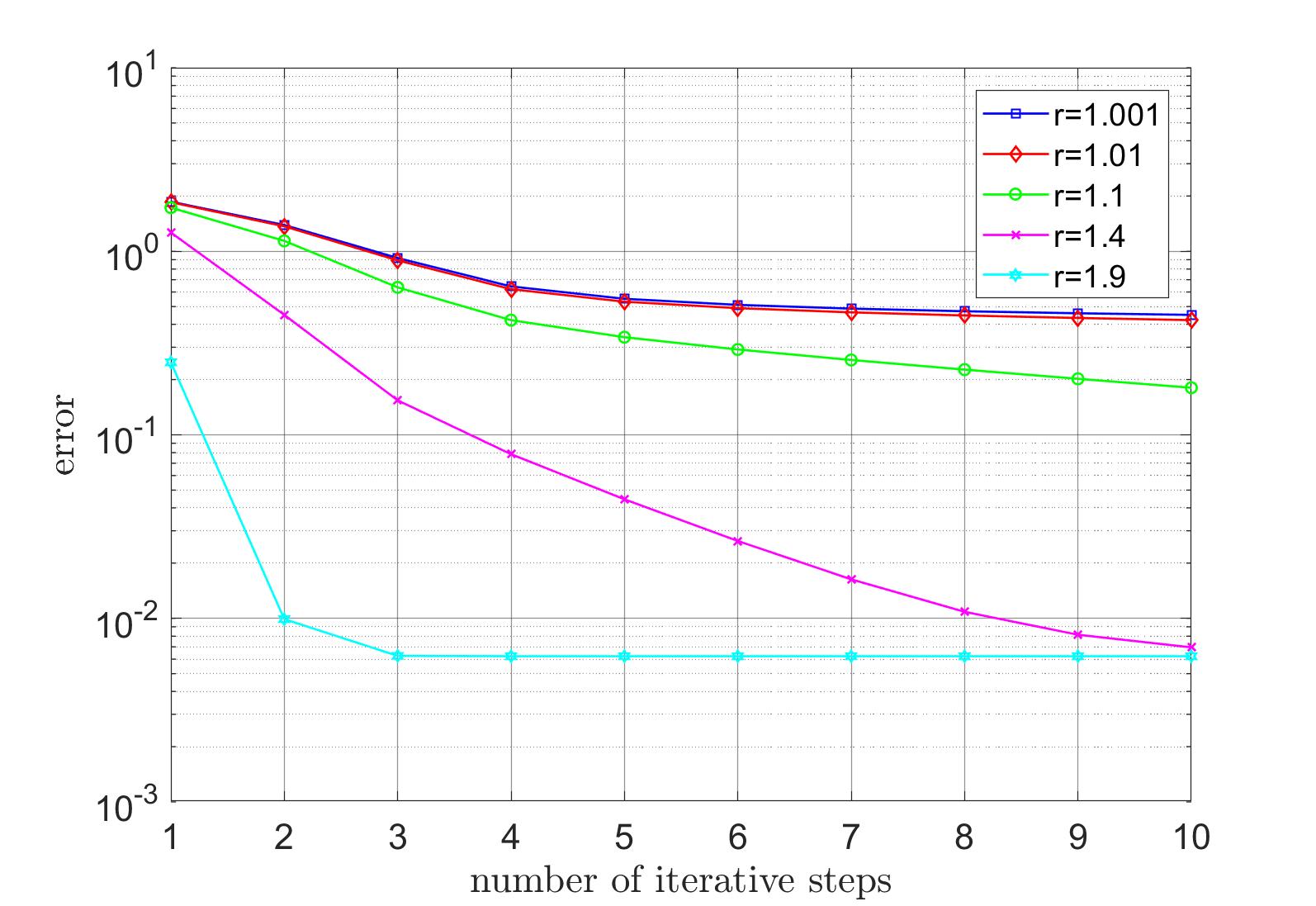}}
 \caption{Influence of the physical parameter $r$ for almost constant viscosity. Left: Carreau model. Right: Relaxed power-law model.}\label{fig:SQ}
\end{figure}


\subsection{Error decay in dependence on the zero and infinite shear rates} Next, we will show that, for fixed $r \in (1,2)$, the convergence rate does not essentially deteriorate when the ratio of the infinite and zero shear rates decreases. As in the experiment before, we choose the source term $f$ so that the unique solution is given by the smooth function $u^\star(x,y)=\sin(\pi x)\sin(\pi y)$. For the Carreau model we set $\lambda=2$, $r=1.5$, $\mu_0=10^a$, and $\mu_\infty=10^{-a}$ for $a \in \{1,2,3,4,5\}$. As we can see from Figure~\ref{fig:CarreauShear} (left), the convergence rate is (almost) independent of $a$, i.e., the convergence does not deteriorate for a decreasing quotient $\nicefrac{\mu_\infty}{\mu_0}$. For the relaxed power-law model we set $r=1.5$, $\varepsilon_{-}=10^{-a}$, and $\varepsilon_{+}=10^a$ for $a \in \{1,2,3,4,5\}$. Even though the plots differ for the various values of $a$, the convergence \emph{rate} is almost the same for all of them; indeed, no significant deterioration of the convergence rate can be observed in Figure~\ref{fig:CarreauShear} (right) for increasing $a$. 

\begin{figure} 
{\includegraphics[width=0.48\textwidth]{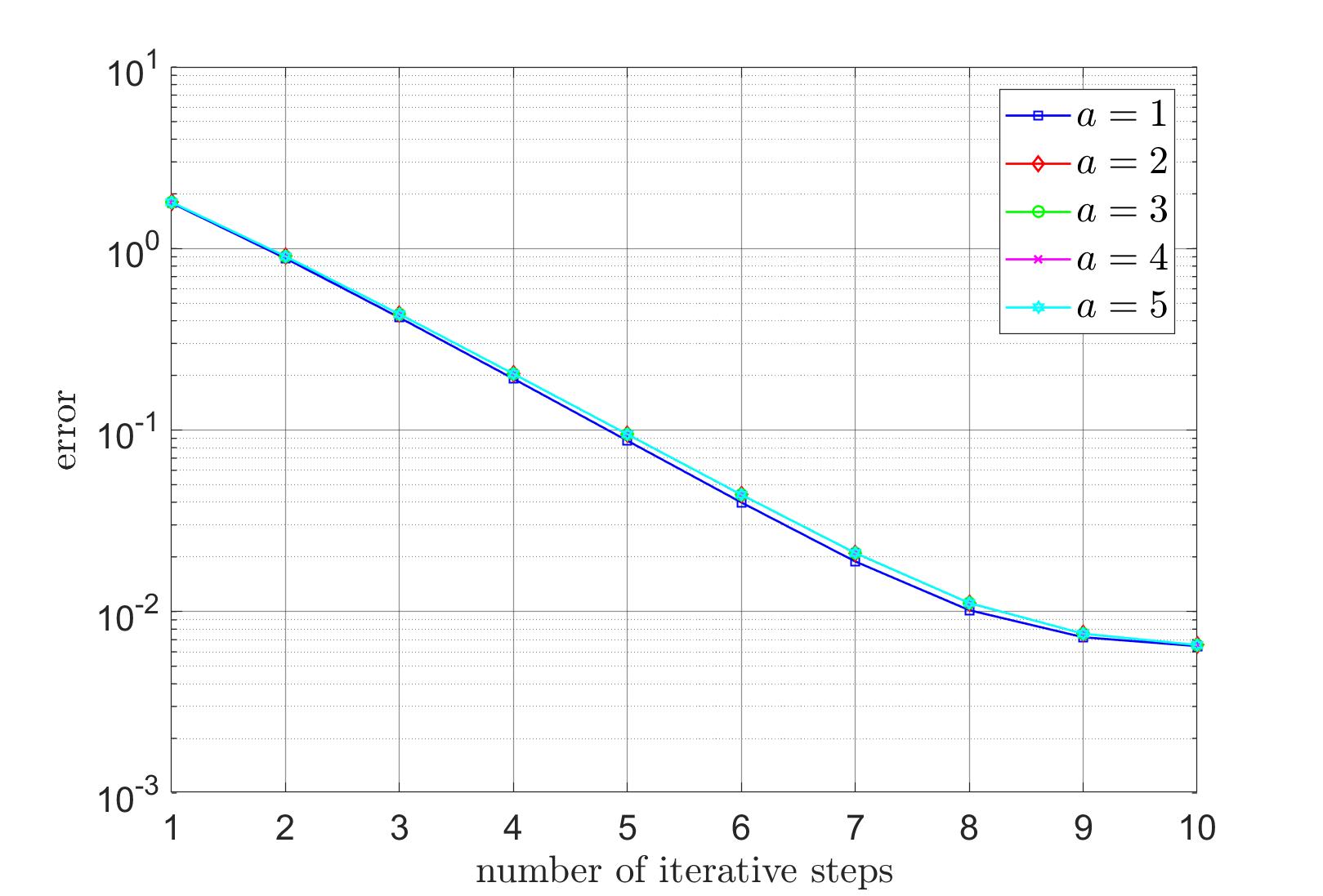}}\hfill
{\includegraphics[width=0.48\textwidth]{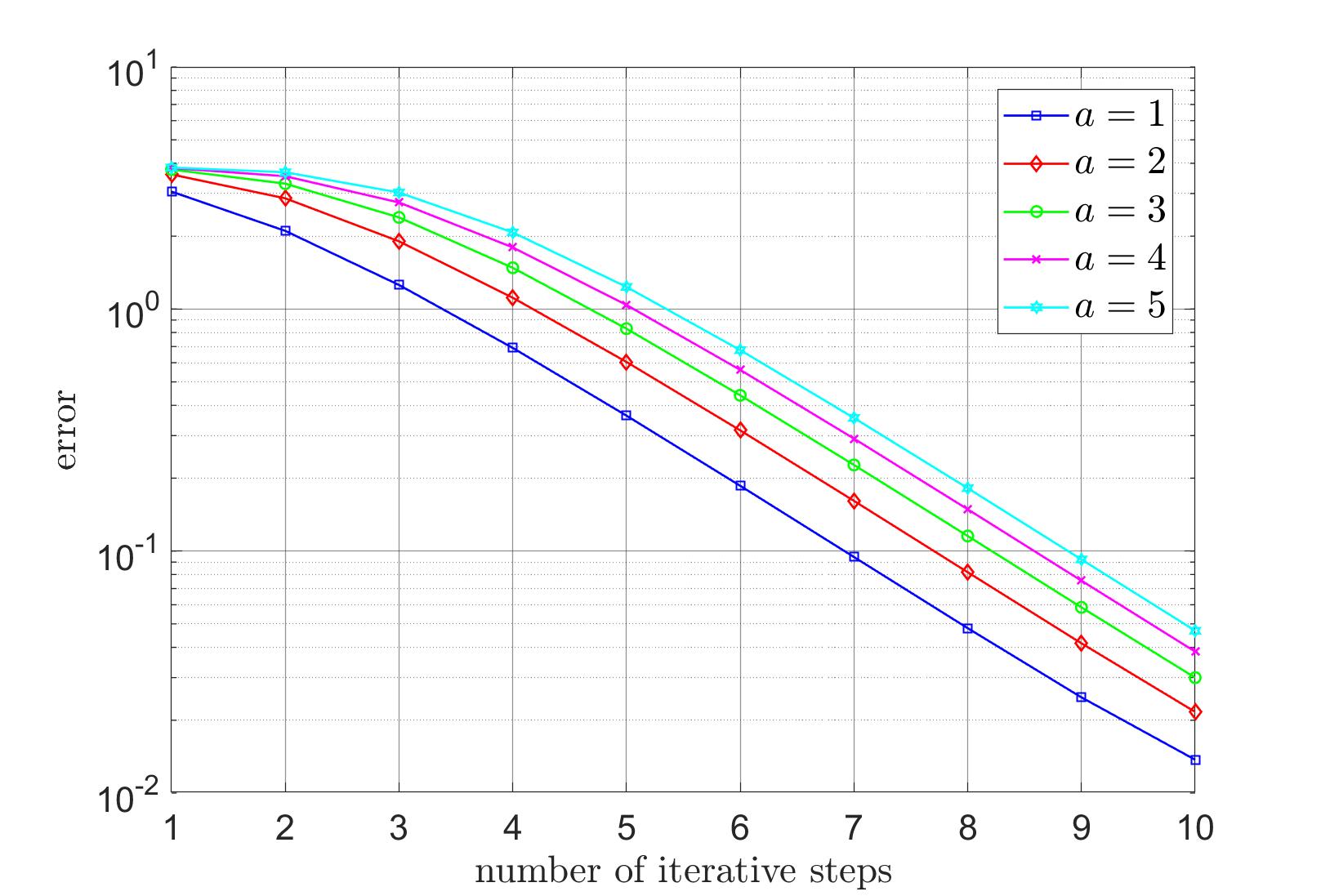}}
 \caption{Influence of the ratio of the infinite and zero shear rates on the convergence. Left: Carreau model with $\lambda=2$, $r=1.5$, $\mu_0=10^a$, and $\mu_\infty=10^{-a}$. Right: Relaxed power-law model with $r=1.5$, $\varepsilon_{-}=10^{-a}$, and $\varepsilon_{+}=10^a$.}\label{fig:CarreauShear}
\end{figure}

\subsection{Energy decay and the contraction factor}

We now focus on the energy decay, and compare the exact contraction factor, cf.~\eqref{eq:qexakt}, the a posteriori computable factor~\eqref{eq:qcomp}, and the worst case factor from Remark~\ref{rem:boundconfac}, cf.~\eqref{eq:qwc}. Again, this will be done for the Carreau and the relaxed power-law models. 
In our figures below, we plot the energy decay $\E(u^n)-\E(u^\star)$, as well as the aforementioned factors 
\begin{align}
q_E(n)&=\frac{\E(u^n)-\E(u^\star)}{\E(u^{n-1})-\E(u^\star)}, \label{eq:qexakt} \\
q_A(n)&=1-\frac{1}{4} \left\{\esssup_{\x \in \Omega} \frac{\mu(|\ef{\uv^n}|^2)}{2 \mu'(|\ef{\uv^n}|^2) |\ef{\uv^n}|^2+\mu(|\ef{\uv^n}|^2)}\right\}^{-1}, \label{eq:qcomp}\\
q_W(n)&=1-\frac{1}{4} \frac{m_\mu}{M_\mu}, \label{eq:qwc}
\end{align}
against the number of iteration steps $n$.

\subsubsection{Energy contraction for the Carreau model}
We consider the Carreau model, cf.~\eqref{eq:carreau}, for $\mu_\infty=1$, $\mu_0=100$, $\lambda=2$, and $r=1.3$, respectively $r=1.1$. In both cases, we approximate the discrete solution for the source term $f$ from case (a) before by performing seventy steps of the Ka\v{c}anov iteration~\eqref{eq:kacanov}, and subsequently use this approximation for the determination of the reference energy $\E(u^\star)$; here, $u^\star$ denotes the unique minimiser in the finite element space. We can clearly observe in Figure~\ref{fig:CarreauContraction} that, on the one hand, the \emph{a posteriori} computable factor $q_A(n)$, cf.~\eqref{eq:qcomp}, is much larger than the actual factor $q_E(n)$, cf.~\eqref{eq:qexakt}. On the other hand, however, the factor $q_A(n)$ clearly still improves the worst case factor $q_W(n)$ from Remark~\ref{rem:boundconfac}, cf.~\eqref{eq:qwc}.

\begin{figure} 
{\includegraphics[width=0.48\textwidth]{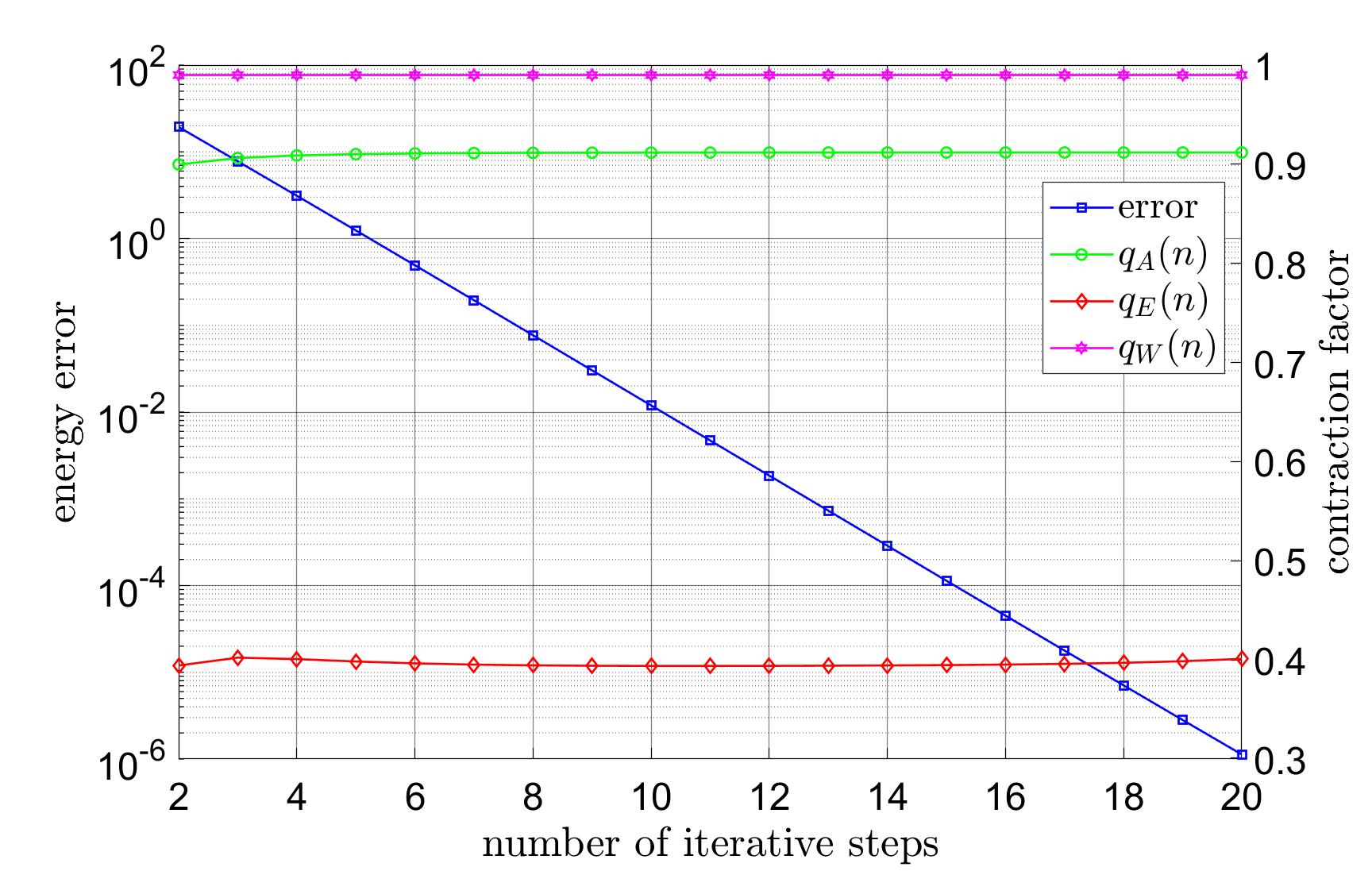}}\hfill
{\includegraphics[width=0.48\textwidth]{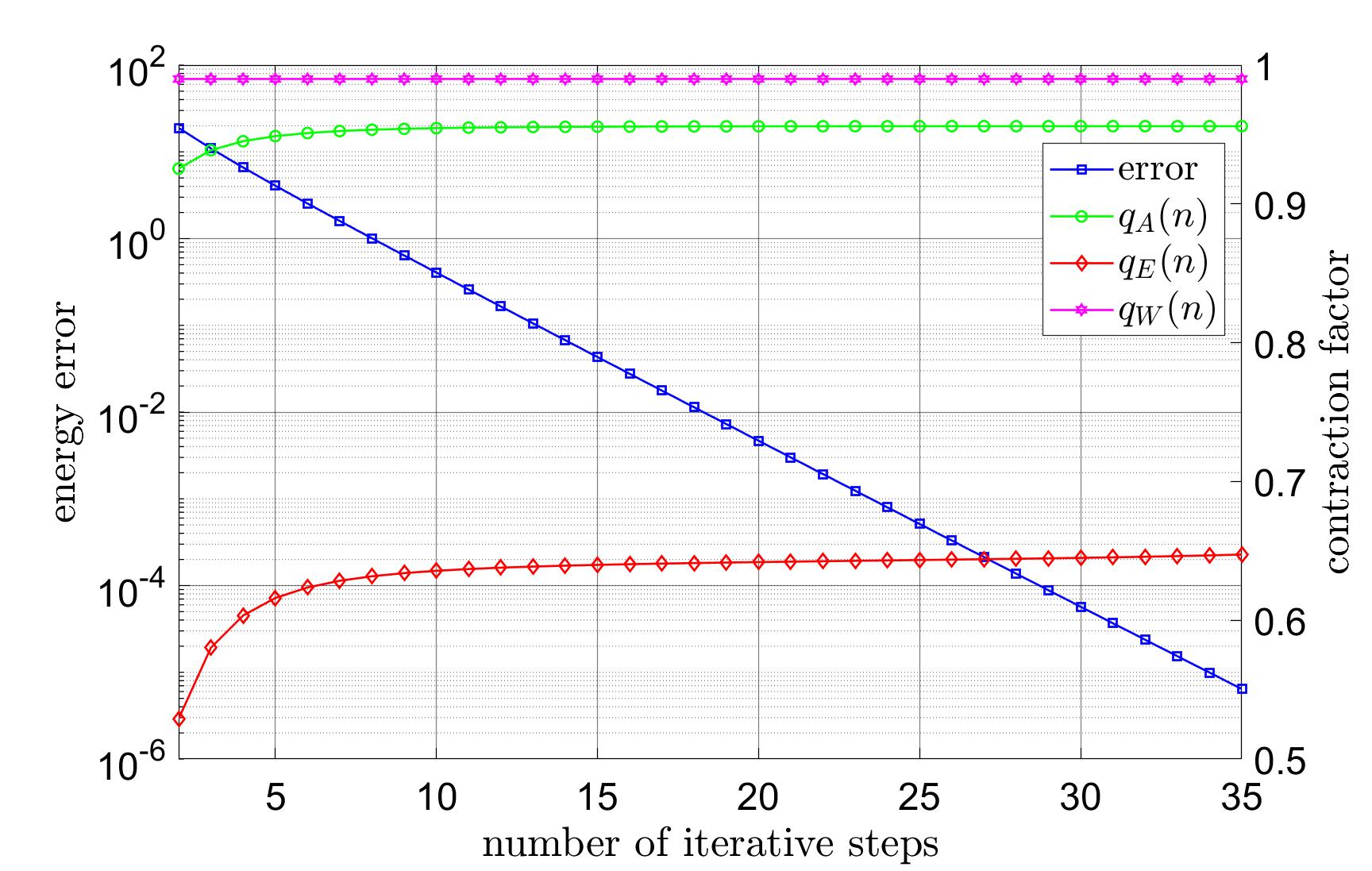}}
 \caption{Energy decay and the contraction factors for the Carreau model with $r=1.3$ (left) and $r=1.1$ (right).}\label{fig:CarreauContraction}
\end{figure}

\subsubsection{Energy contraction for the relaxed power-law model}
Let the coefficient $\mu$ obey the relaxed power-law model with $\varepsilon_{-}=10^{-6}$, $\varepsilon_{+}=10^6$, and $r=1.3$, respectively $r=1.1$. In this experiment, we approximate the discrete solution $u^\star$ by performing fifty and one hundred iteration steps for $r=1.3$ and $r=1.1$, respectively. As before, the a posteriori computable contraction factor $q_A(n)$ is noticeably larger than the exact factor $q_E(n)$, however, this is less marked than before; see Figure~\ref{fig:PowerLawContraction}. Moreover, it considerably improves the worst case contraction factor $q_W(n) \approx 1-10^{-12}$. 

\begin{figure} 
{\includegraphics[width=0.48\textwidth]{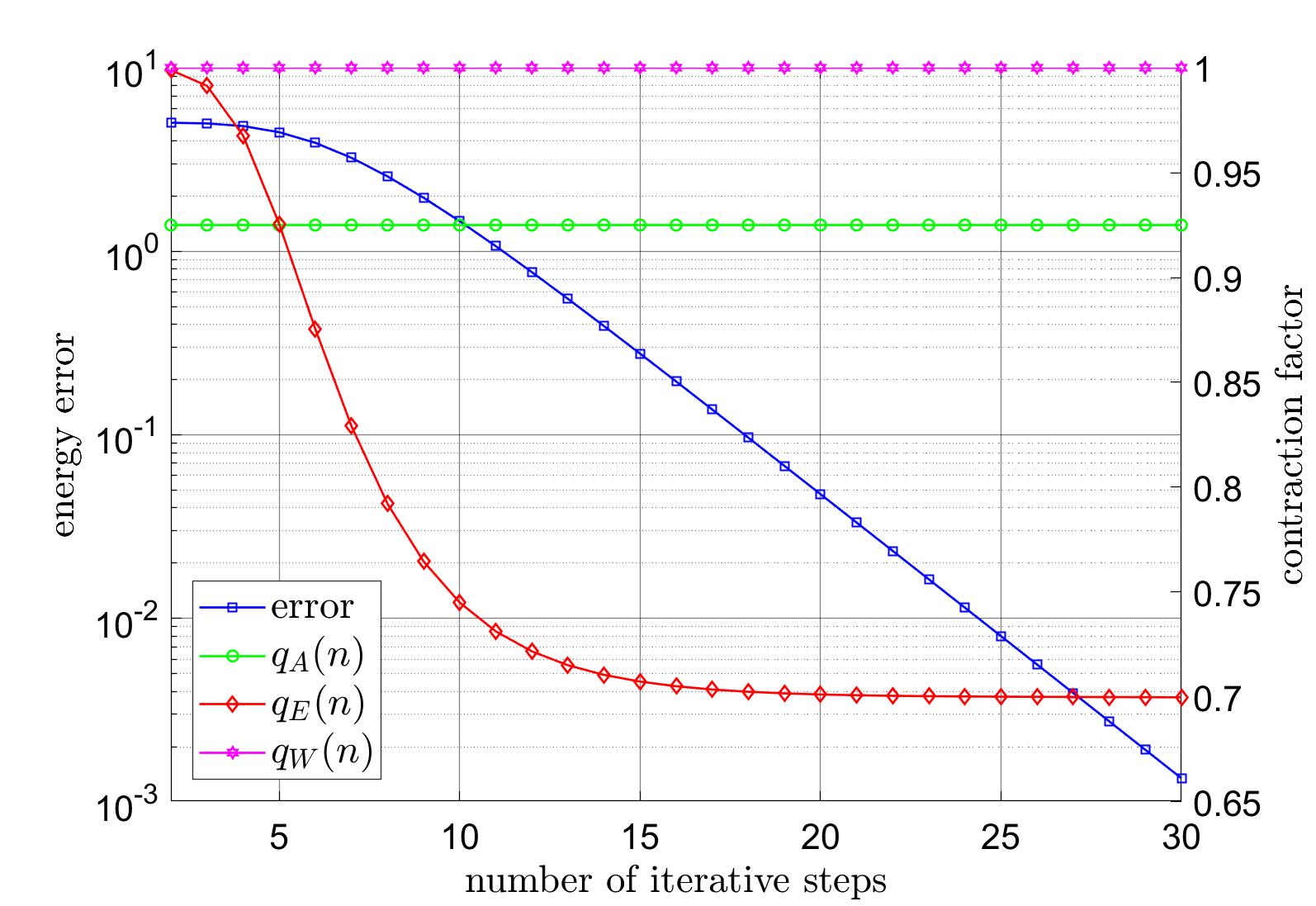}}\hfill
{\includegraphics[width=0.48\textwidth]{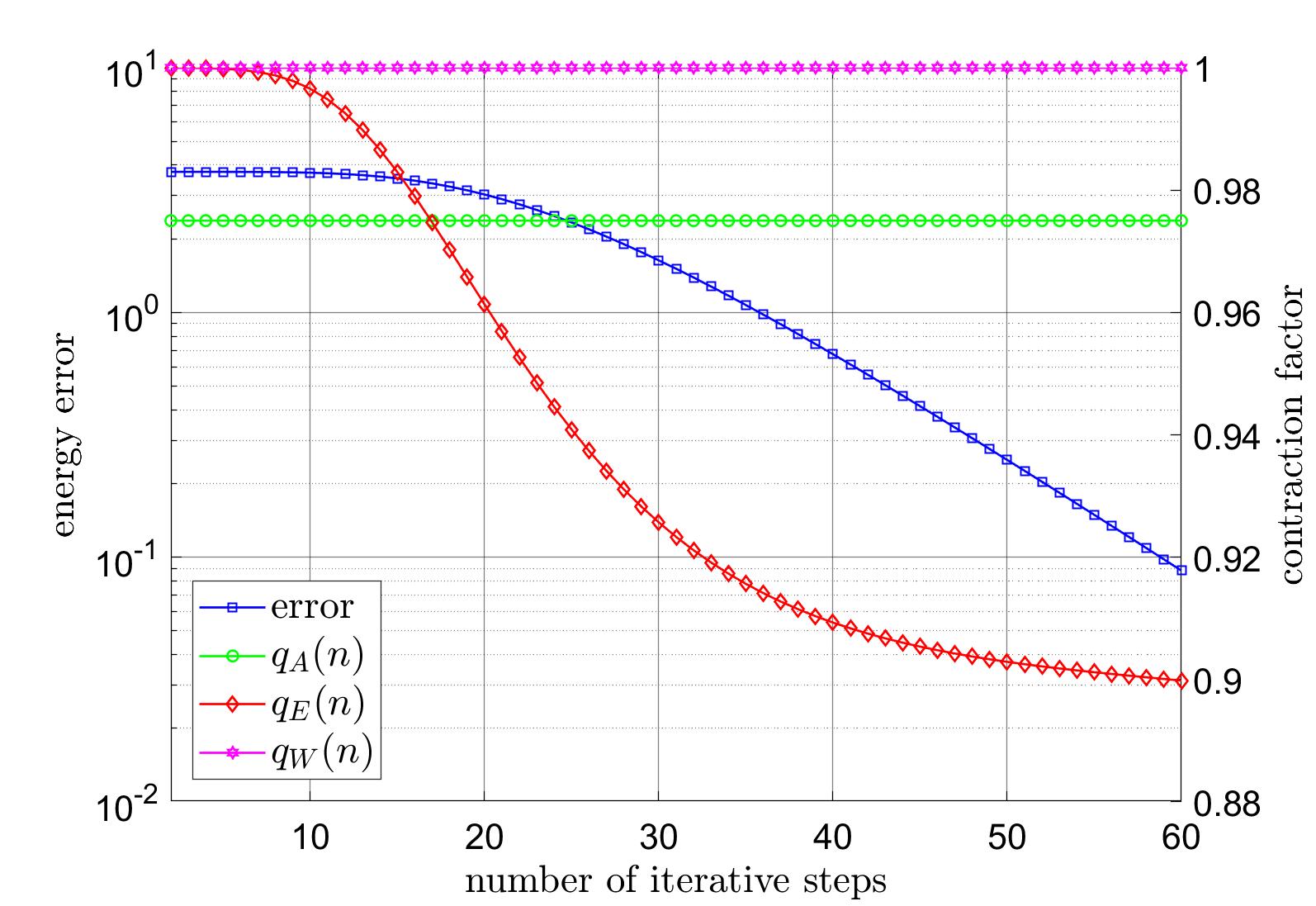}}
 \caption{Energy decay and the contraction factors for the power-law model with $r=1.3$ (left) and $r=1.1$ (right).}\label{fig:PowerLawContraction}
\end{figure}

We now repeat this experiment on a coarser mesh consisting of $\mathcal{O}(10^5)$ uniform triangles. In Figure~\ref{fig:PowerLawContraction2} we plot the factors $q_A(n)$, $q_E(n)$, as well as $q(n)$ from~\eqref{eq:contractionfactor} against the number of iteration steps. We observe that the (non-computable) factor $q(n)$ from~\eqref{eq:contractionfactor} has a similar trend as the exact factor $q_E(n)$, cf.~\eqref{eq:qexakt}, and approximates the computable factor~$q_A(n)$, cf.~\eqref{eq:qcomp}, as the number of iteration steps increases.

\begin{figure} 
{\includegraphics[width=0.48\textwidth]{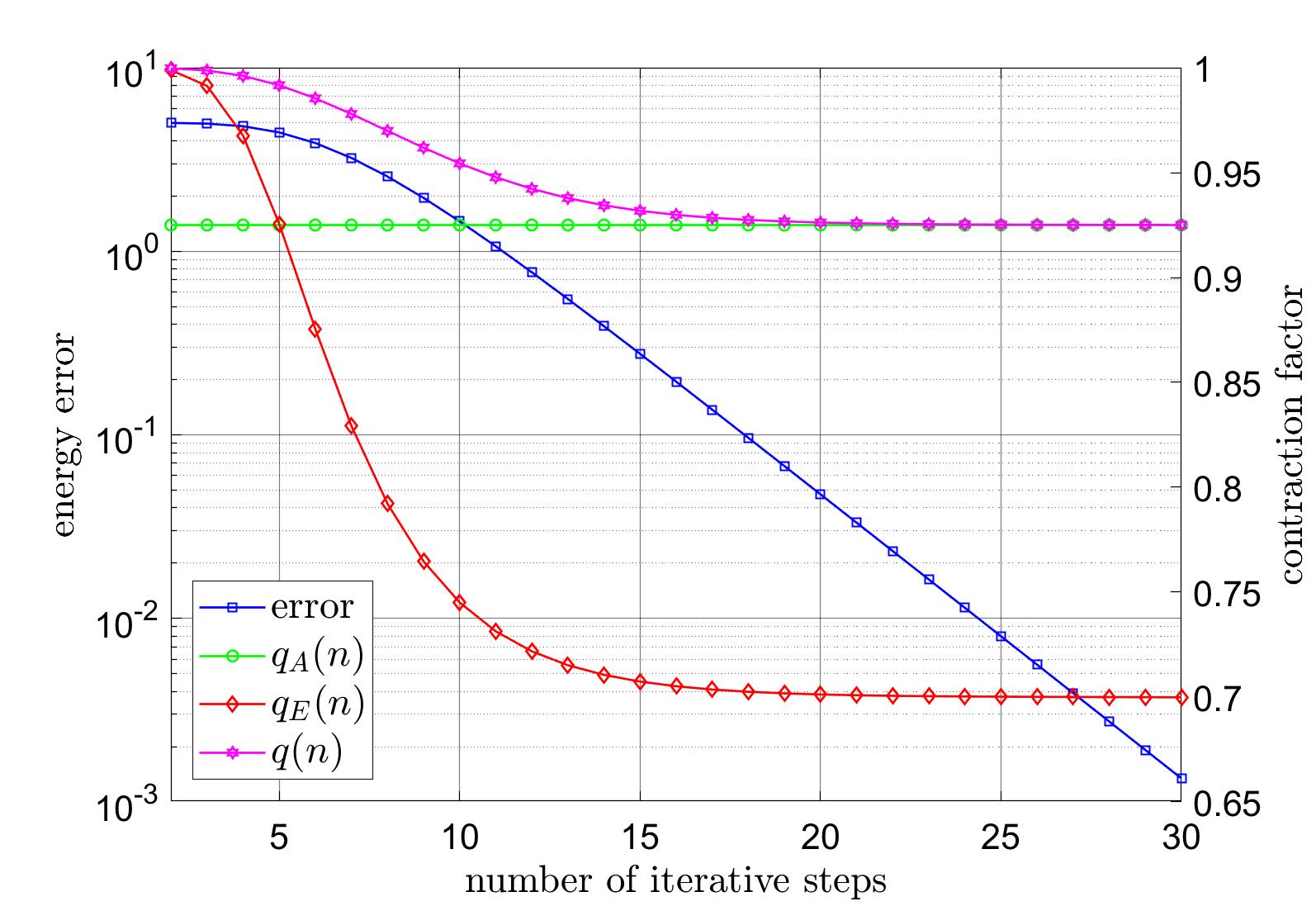}}\hfill
{\includegraphics[width=0.48\textwidth]{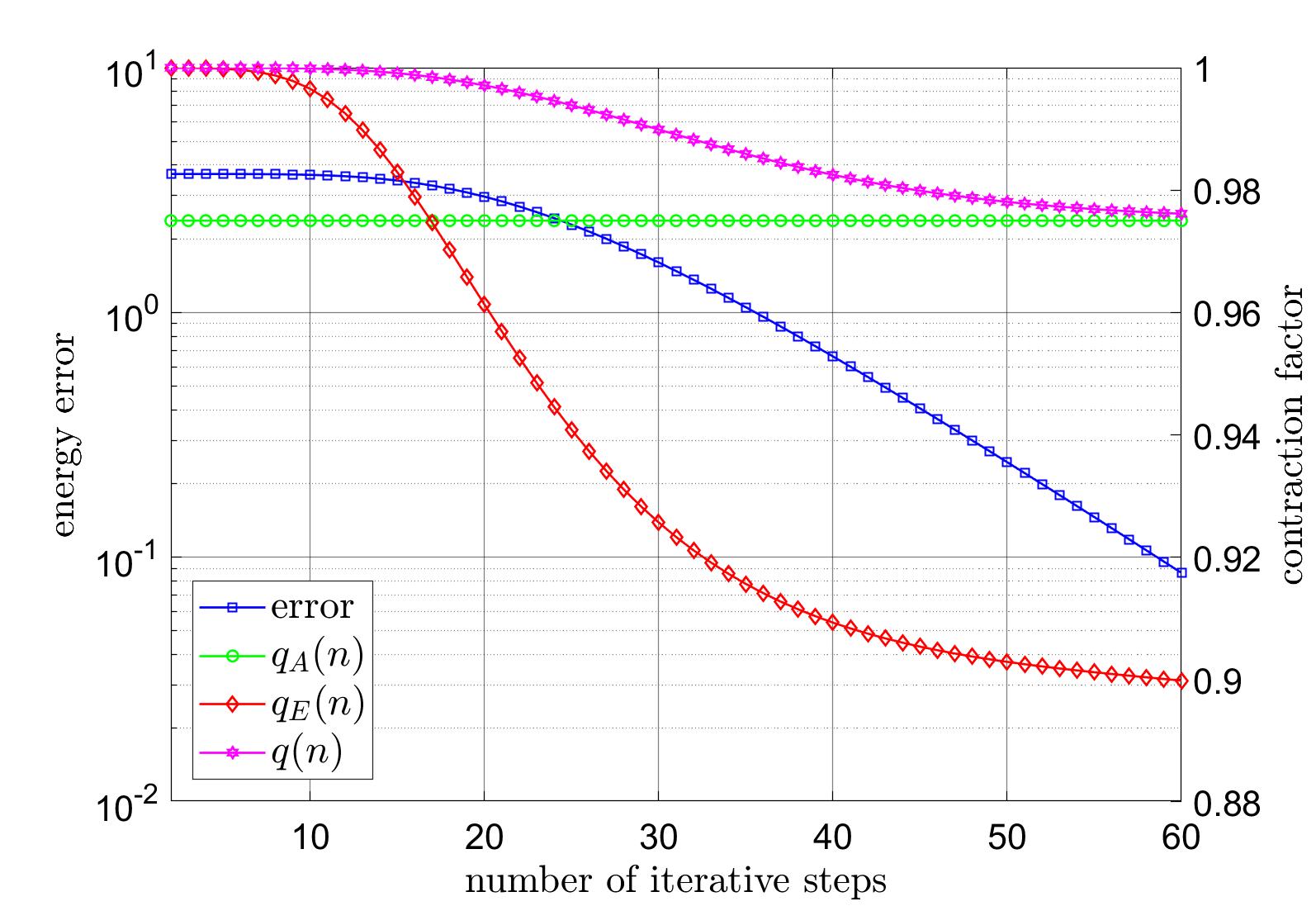}}
 \caption{Energy decay and the contraction factors for the power-law model with $r=1.3$ (left) and $r=1.1$ (right) in the coarser mesh.}\label{fig:PowerLawContraction2}
\end{figure}

Finally, we remark that, in the context of fixed point iterations, the contraction factor can be (heuristically) approximated by 
\begin{align} \label{eq:qh}
q_H(n):=\min\left\{1,\frac{\E(u^n)-\E(u^{n-1})}{\E(u^{n-1})-\E(u^{n-2})}\right\}
\end{align}
as $n \to \infty$, see, e.g.,~\cite{senning:2007}; we emphasize that $q_H(n) \geq 0$, for $n \geq 2$, thanks  to~\eqref{eq:thmfirstsummand}. As can be observed in Figure~\ref{fig:PowerLawContraction3}, the factor $q_H(n)$, cf.~\eqref{eq:qh}, does indeed approximate the exact factor $q_E(n)$ from~\eqref{eq:qexakt} well for sufficiently large $n$. However, in contrast with the bound $q_A(n)$ from Theorem~\ref{thm:globalcontractioncomp}, the computable factor $q_H(n)$ does not provide any guaranteed \textit{a priori} information.

\begin{figure} 
{\includegraphics[width=0.48\textwidth]{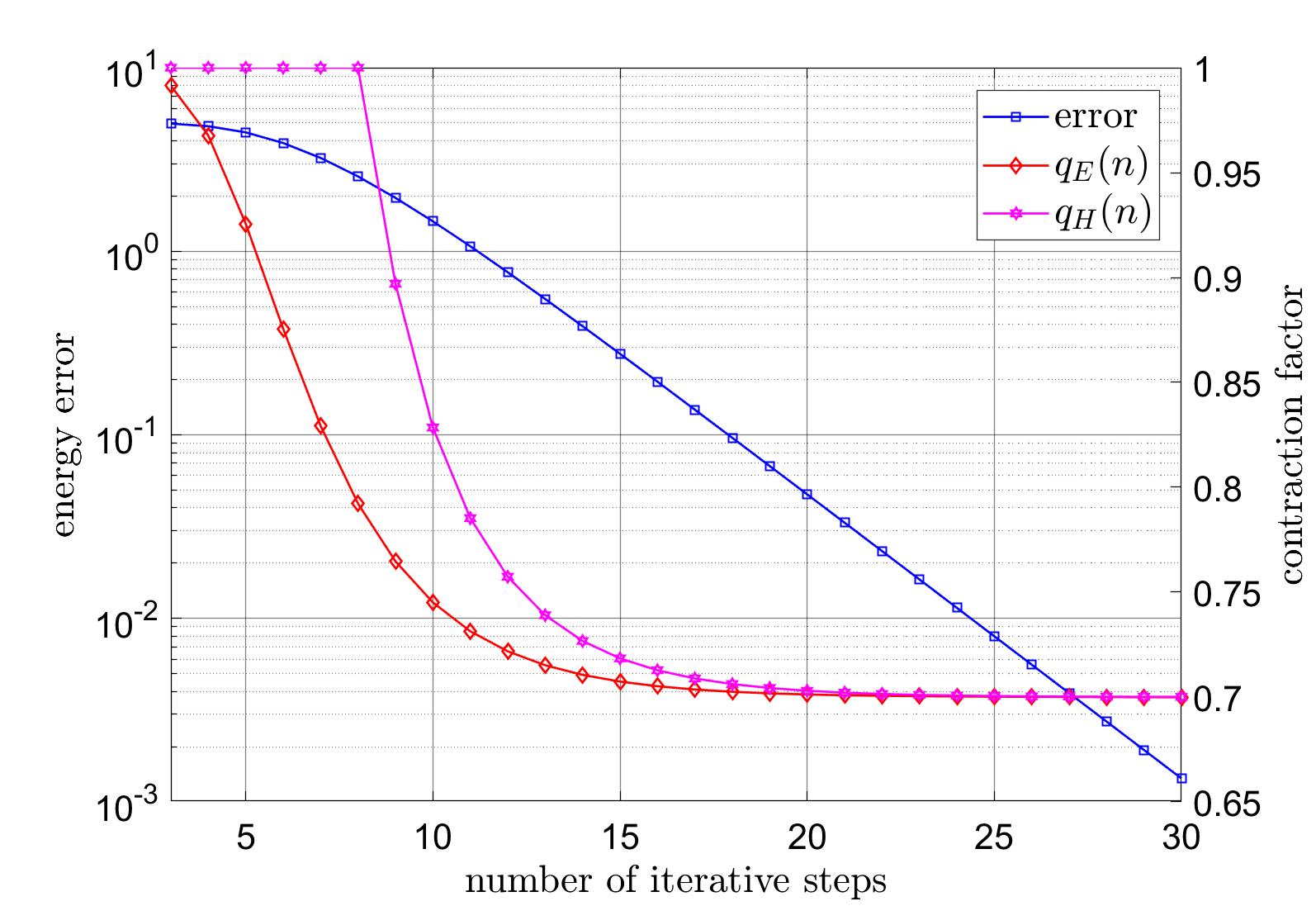}}\hfill
{\includegraphics[width=0.48\textwidth]{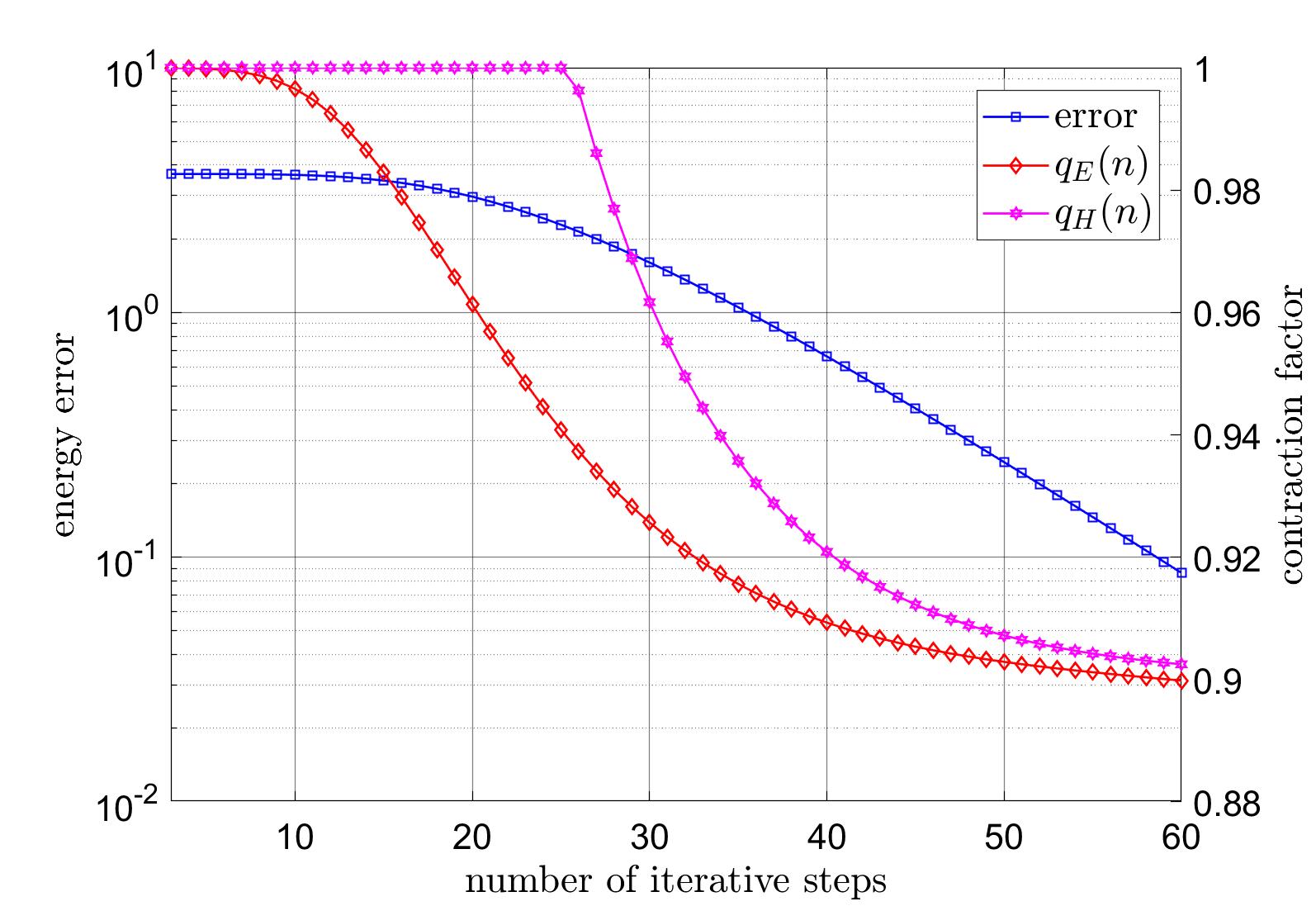}}
 \caption{Energy decay and the contraction factors for the power-law model with $r=1.3$ (left) and $r=1.1$ (right).}\label{fig:PowerLawContraction3}
\end{figure}


\subsection{Energy decay for different mesh sizes}

We conclude this section with a comparison of the energy decay for different mesh sizes. For the Carreau model, cf.~\eqref{eq:carreau}, we set $\mu_\infty=1$, $\mu_0=100$, $\lambda=2$, and $r=1.3$. In the case of the relaxed power-law model, let $\varepsilon_{-}=10^{-6}$, $\varepsilon_{+}=10^6$, and $r=1.3$. In each case we approximated the discrete solution, and, in turn, the corresponding energy by performing one hundred iteration steps. As we can see from Figure~\ref{fig:MeshComp}, the asymptotic convergence rates (almost) coincide for the different mesh sizes.

\begin{figure} 
{\includegraphics[width=0.48\textwidth]{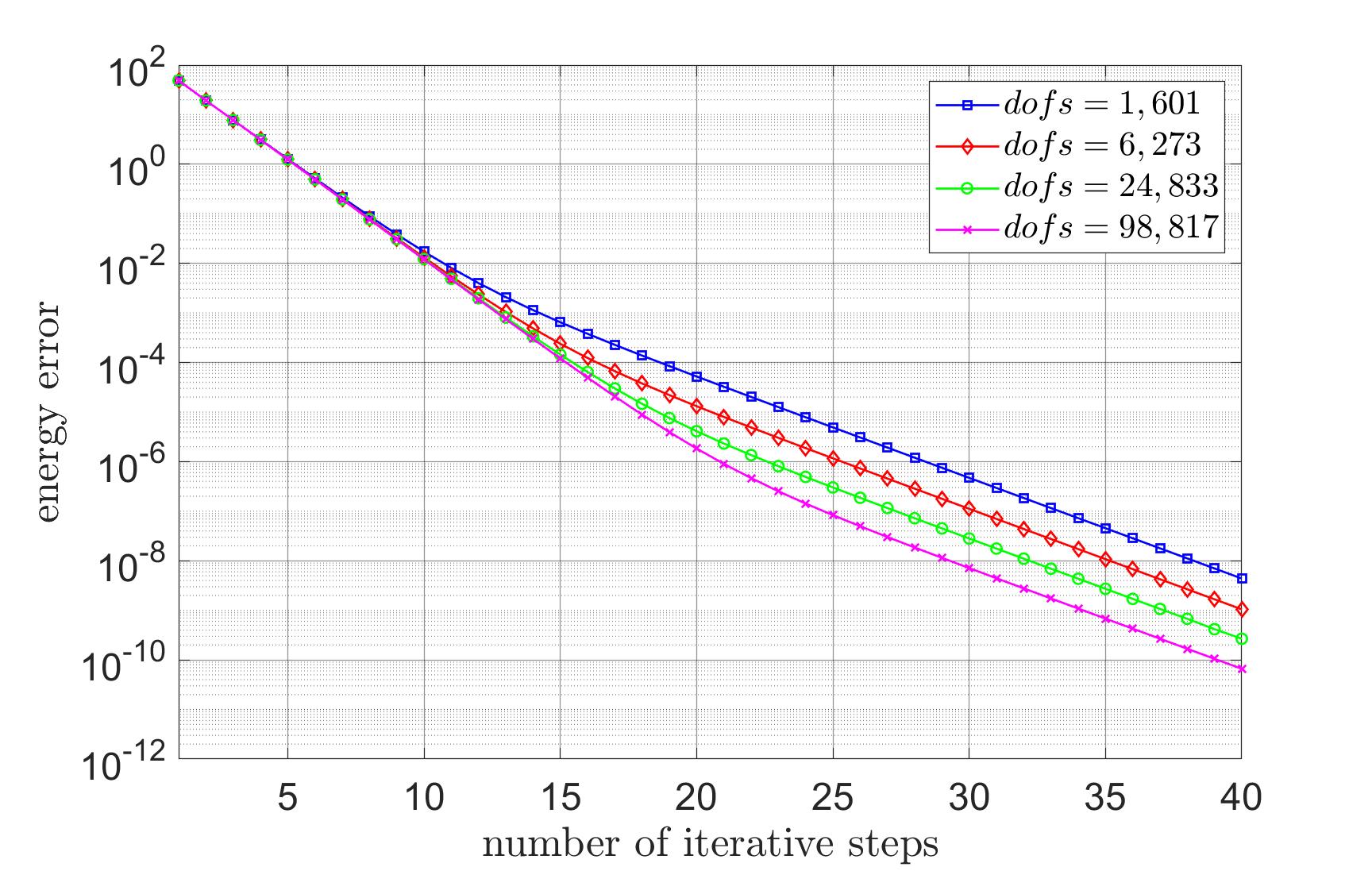}}\hfill
{\includegraphics[width=0.48\textwidth]{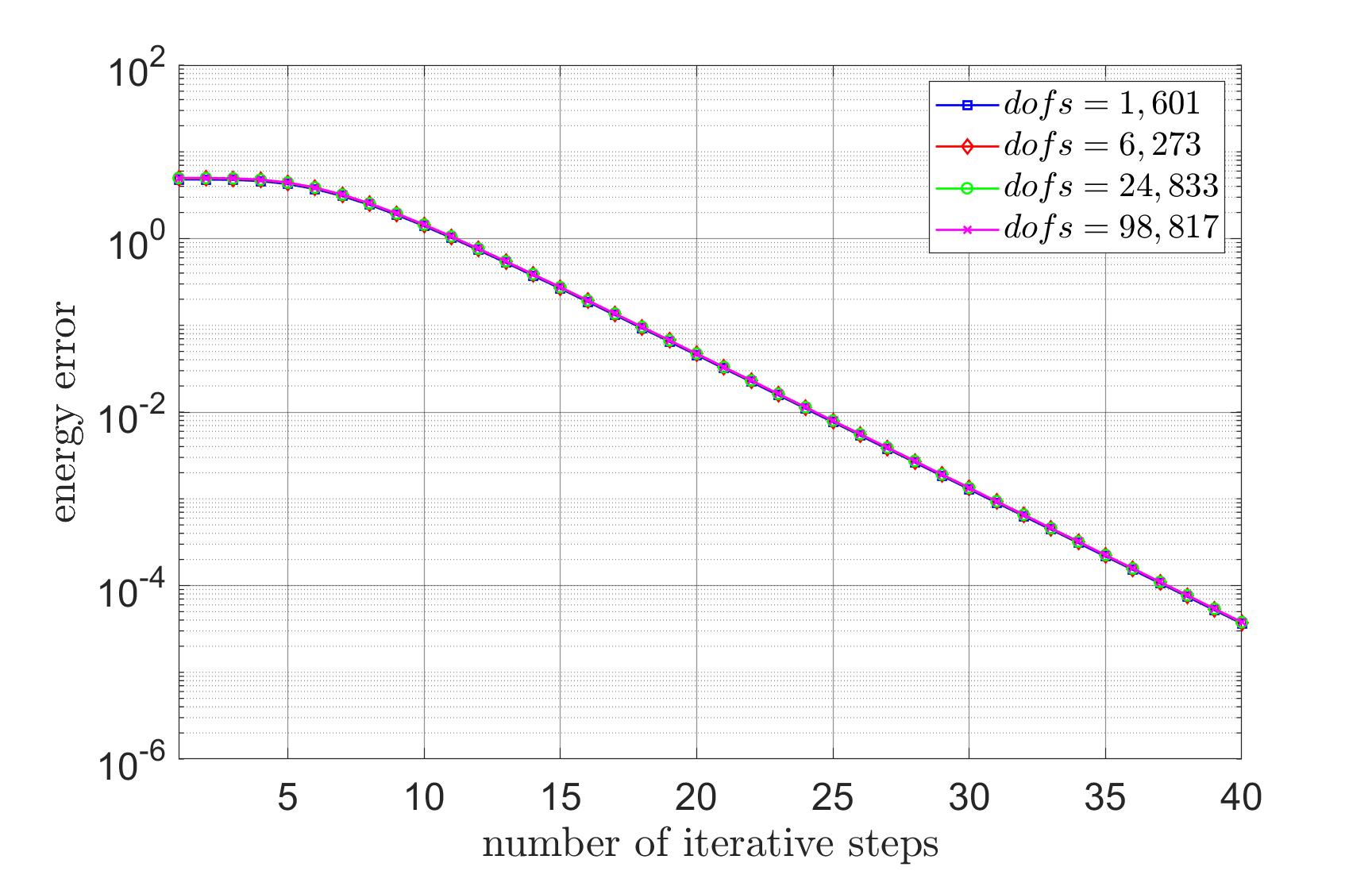}}
 \caption{Energy decay for the Carreau model (left) and relaxed power-law model (right) for different mesh sizes.}\label{fig:MeshComp}
\end{figure}

\section{Conclusion} \label{sec:conclusions}
In this article, we established an \emph{a posteriori} computable (energy) contraction factor for the Ka\v{c}anov scheme (on finite-dimensional Galerkin spaces), motivated by applications to quasi-Newtonian fluid flow problems. For the relaxed power-law model, this factor is independent of the relaxation parameters $\varepsilon_{\pm}$; we also demonstrated that it is, instead, the power-law exponent that affects the convergence rate of the iteration. In contrast, existing bounds on the contraction factor of the relaxed Ka\v{c}anov iteration depend on the relaxation parameters $\varepsilon_{\pm}$ in an unfavourable manner, in the sense that they tend to 1 as $\varepsilon_{-} \to 0$ and/or $\varepsilon_{+} \to \infty$. A series of numerical tests have confirmed that our \emph{a posteriori} computable contraction factor improves, on finite-dimensional Galerkin spaces, existing bounds, and that, as predicted by our analysis, for the power-law model it is in fact the closeness of the power-law exponent $r \in (1,2)$ to 1 that influences the convergence rate of the iteration. However, our experiments revealed that the theoretically derived bound on the contraction factor of the Ka\v{c}anov scheme is still too pessimistic.   

\begin{appendix}
\section{Smoothly relaxed power-law model} \label{app:differentiablerelaxatio}
In this appendix, we will introduce a continuously differentiable approximation of the relaxed power-law viscosity~\eqref{eq:relaxedpowerlaw}. In particular, for $0<\delta <\varepsilon_{-}^2 < \varepsilon_{+}^2$, we want to define a function $\mu_{\delta,\varepsilon}:\mathbb{R}_{\geq 0} \to \mathbb{R}_{\geq 0}$ which
\begin{enumerate}[(a)]
\item is continuously differentiable;
\item coincides with $\mu_{\varepsilon}$, cf.~\eqref{eq:relaxedpowerlaw}, in the domain $[\varepsilon_{-}^2+\delta,\varepsilon_{+}^2-\delta]$; 
\item is constant on $[0,\varepsilon_{-}^2-\delta] \cup [\varepsilon_{+}^2+\delta,\infty)$;
\item converges pointwise to $\mu_\varepsilon$ for $\delta \to 0$. 
\end{enumerate}
The idea is to identify quadratic functions $g_{\delta,\varepsilon}^{\pm}$ on $[\varepsilon_{\pm}^2-\delta,\varepsilon_{\pm}^2+\delta]$, respectively, which smoothly connect the constant parts of $\mu_{\varepsilon}$ with the map $t \mapsto t^{\frac{r-2}{2}} = \mu(t)$ on $[\varepsilon_{-}^2+\delta,\varepsilon_{+}^2-\delta]$, i.e.,
\begin{align*}
\mu_{\delta,\varepsilon}(t)=\begin{cases}
g_{\delta,\varepsilon}^{-}(\varepsilon_{-}^2-\delta) & 0 \leq t \leq \varepsilon_{-}^2-\delta \\
g_{\delta,\varepsilon}^{-}(t) & \varepsilon_{-}^2-\delta<t\leq \varepsilon_{-}^2+\delta \\
t^{\frac{r-2}{2}} & \varepsilon_{-}^2+\delta < t \leq \varepsilon_{+}^2-\delta \\
g_{\delta,\varepsilon}^{+}(t) & \varepsilon_{+}^2-\delta <t \leq \varepsilon_{+}^2+\delta \\
g_{\delta,\varepsilon}^{+}(\varepsilon_{+}^2+\delta) & t > \varepsilon_{+}^2+\delta.
\end{cases}
\end{align*}
A straightforward calculation reveals that the properties (a)--(d) are satisfied for
\begin{align*}
g_{\delta,\varepsilon}^{-}(t)=&\frac{(\varepsilon_{-}^2+\delta)^{\frac{r-4}{2}}(r-2)}{8 \delta} t^2-\frac{(\varepsilon_{-}^2+\delta)^{\frac{r-4}{2}}(\varepsilon_{-}^2- \delta)(r-2)}{4 \delta}t \\
& \quad -\frac{(\varepsilon_{-}^2+ \delta)^{\frac{r-2}{2}}(-14 \delta + 2 \varepsilon_{-}^2+3 \delta r-\varepsilon_{-}^2r)}{8 \delta}
\end{align*}
and 
\begin{align*}
g_{\delta,\varepsilon}^{+}(t)=&-\frac{(\varepsilon_{+}^2-\delta)^{\frac{r-4}{2}}(r-2)}{8 \delta} t^2+\frac{(\varepsilon_{+}^2-\delta)^{\frac{r-4}{2}}(\varepsilon_{+}^2+\delta)(r-2)}{4 \delta} t \\
& \quad -\frac{(\varepsilon_{+}^2-\delta)^{\frac{r-2}{2}}(-14 \delta - 2 \varepsilon_{+}^2+3 \delta r+\varepsilon_{+}^2r)}{8 \delta}.
\end{align*}

\begin{figure} 
{\includegraphics[width=0.48\textwidth]{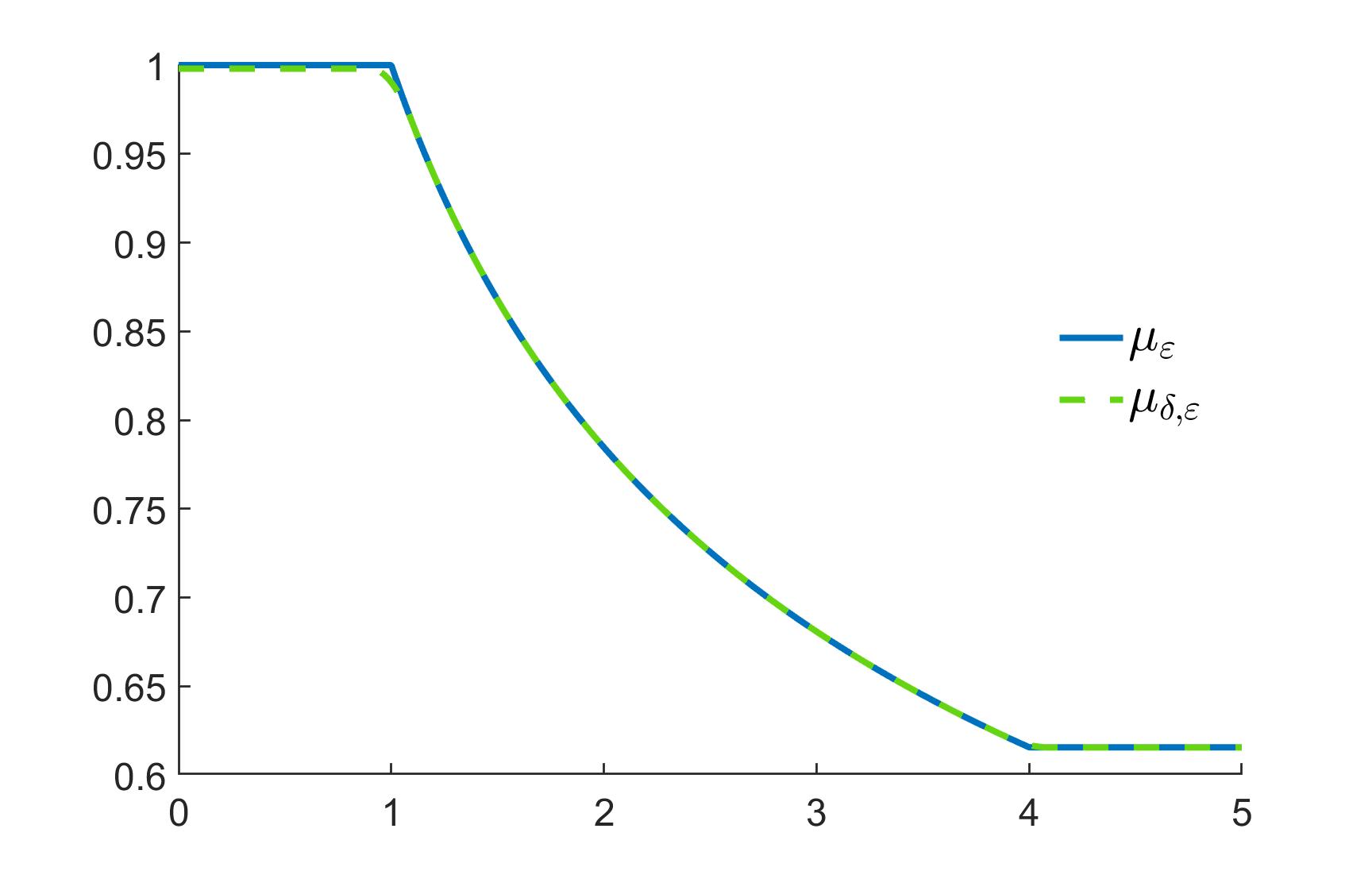}}
 \caption{Comparison of $\mu_{\varepsilon}$ and $\mu_{\delta,\varepsilon}$ for $r=1.3, \, \varepsilon_{-}=1, \, \varepsilon_{+}=2$, and $\delta=0.1$.}
\end{figure}
\end{appendix}

\bibliographystyle{amsplain}
\bibliography{references} 
\end{document}